\documentclass[11pt,a4paper]{amsart} %%%
\usepackage[english]{babel}
\usepackage{amssymb,xspace,graphics,dsfont,enumerate,latexsym,amsfonts,amsbsy,amstext, mathabx, stmaryrd,mathtools}
\usepackage[pagebackref,colorlinks=true,linkcolor=blue,citecolor=blue,urlcolor=red, hypertexnames=true]{hyperref}
\usepackage[abbrev,backrefs,nobysame]{amsrefs}
\usepackage[mathscr]{euscript}
\theoremstyle{plain}

\newtheorem{theorem}{Theorem}[section]
\newtheorem{lemma}[theorem]{Lemma}
\newtheorem{claim}[theorem]{Claim}
\newtheorem{fact}[theorem]{Fact}
\newtheorem{property}[theorem]{Property}
\newtheorem{proposition}[theorem]{Proposition}
\newtheorem{corollary}[theorem]{Corollary}
\numberwithin{equation}{section}
{\theoremstyle{definition}\newtheorem{remark}[theorem]{Remark}}

\def\cc{\mathbb{C}}

\def\nn{\mathbb{N}}

\def\R{\ensuremath{\mathbb R}}

\def\C{\ensuremath{\mathbb C}}

\def\N{\ensuremath{\mathbb N}}

\newcommand{\ti}[1]{\widetilde{#1}}

\def\ca{{\mathcal{A}}}
\def\cb{{\mathcal{B}}}
\def\cg{{\mathcal{G}}}
\def\cn{{\mathcal {N}}}
\def\cl{{\mathcal{L}}}

\def\cu{{\mathcal {U}}}
\def\cv{{\mathcal {V}}}
\def\cw{{\mathcal {W}}}

\usepackage[all]{xy}
\entrymodifiers={+!!<0pt,\fontdimen22\textfont2>}

\newcommand{\flba}[2]{
\xymatrix@C15pt{#1\ar@{|->}[r]&#2}}
\newcommand{\flcourte}[2]{
\xymatrix@C12pt{#1\ar[r]&#2}}

\newcommand{\spectre}[1]{\sigma\left(#1\right)}
\newcommand{\oc}{H(\cc)}
\newcommand{\vertiii}[1]{{\left\vert\kern-0.25ex\left\vert\kern-0.25ex\left\vert #1 
    \right\vert\kern-0.25ex\right\vert\kern-0.25ex\right\vert}}
\newcommand{\diffp}[2]{\frac{\partial #1}{\partial #2}}

\newcommand{\sot}{\texttt{SOT}{}}
\newcommand{\clmk}{\cl_{(M_{j}),(k_{j})}}
\newcommand{\vspan}[1]{\mathrm{span}\left[#1\right]}
\newcommand{\sote}{\texttt{SOT}\mbox{$^{*}$}}

\newcommand{\bmx}{\mathcal{B}_{M}(X)}
\newcommand{\bx}{\mathcal{B}(X)}

\newcommand{\hct}{\mathrm{HC}(T)}

\usepackage[shortlabels]{enumitem}

\newcommand{\entireFns}{H(\C)}  %% space of entire functions

\newcommand{\abs}[1]{\left| #1 \right|}
\newcommand{\norm}[1]{{\left\|#1\right\|}}

\newcommand{\essSpec}[1]{\sigma_\mathrm{e}\left( #1 \right)}

\newcommand{\Bmsp}[2]{\mathcal{B}_#1(#2)}  		% space of bounded linear operators with norm up to arg
\newcommand{\Bmx}{\Bmsp{M}{X}}

\newcommand{\sotstar}{\textnormal{\texttt{SOT}}\ensuremath{^*}}

\newcommand{\hc}[1]{HC\left( #1 \right)}  % set of hypercyclic vectors

\DeclareMathOperator{\ran}{ran}	% codimension
\DeclareMathOperator{\codim}{codim}	% codimension

\title[Typicality of hypercyclic algebras]{Typicality of operators on Fr\'echet algebras admitting a hypercyclic algebra}
\author{William Alexandre, Clifford Gilmore, Sophie Grivaux}
\address[W. Alexandre]{Univ.\ Lille, CNRS, UMR 8524 - Laboratoire Paul
Painlev\'e, F-59000 Lille, France}
\address[C. Gilmore]{Universit\'e Clermont Auvergne, Laboratoire de Mathématiques Blaise Pascal UMR 6620,  CNRS, F-63178 Aubi\`ere, France}
\address[S. Grivaux]{Univ.\ Lille, CNRS, UMR 8524 - Laboratoire Paul
	Painlev\'e, F-59000 Lille, France}

\email{william.alexandre@univ-lille.fr}
\email{clifford.gilmore@uca.fr}
\email{sophie.grivaux@univ-lille.fr}

\thanks{All three authors were supported by the Labex CEMPI (ANR-11-LABX-0007-01). This work began when the second author was a CEMPI Postdoctoral Fellow at the University of Lille. We are grateful to Fr\'ed\'eric Bayart and Quentin Menet for useful comments on a first version of this paper, and especially to Fr\'ed\'eric Bayart for pointing out a missing step in the proof of Theorem 1.2.}
\begin{document}
\begin{abstract}
This paper is devoted to the study of typical properties (in the Baire Category sense) of certain classes of continuous linear operators acting on Fr\'echet algebras, endowed with the topology of pointwise convergence.
% , i.e. an algebra of vectors which have a dense orbit under the action of the operator. 
Our main results show that within natural Polish spaces of continuous operators acting on the algebra $\oc$ of entire functions on $\cc$, a typical operator supports a hypercyclic algebra. We also investigate the case of the complex Fr\'echet algebras $X=\ell_{p}(\nn)$, $1\le p<+\infty$, or $X=c_{0}(\nn)$ endowed with the coordinatewise  product, and show that whenever $M>1$, a typical operator on $X$ of norm less than or equal to $M$ admits a hypercyclic algebra.
\end{abstract}

\maketitle

\section{Introduction and main results}\label{Section introduction}
Our main aim in this work is to show  that within certain natural classes of continuous linear operators acting on Fr\'echet algebras, endowed with the topology of pointwise convergence, a typical operator possesses 
a hypercyclic algebra. Its set of hypercyclic vectors thus has one of the richest structures for which one can hope.
\par\smallskip
Let us recall some pertinent definitions: given a continuous operator $T$ on a topological vector space $X$, a vector $x$ is said to be \emph{a hypercyclic vector for} $T$ if its orbit 
$\{T^{n}x\;;\;n\geq0\}$ under the action of $T$ is dense in $X$. The set of hypercyclic vectors for $T$ is denoted by $\hct$, and it is dense in $X$ as soon as it is nonempty. Whenever $X$ is a 
second-countable Baire space, $\hct$ is a $G_{\delta }$ set (i.e.\ a countable intersection of open sets), so it is residual in $X$ as soon as it is nonempty. The study of the linear structure of this set 
$\hct$ has been the object of many interesting and deep studies: it is known that whenever $\hct$ is nonempty, the set $\hct\cup\{0\}$ contains a dense linear manifold \cite{Bou}. For certain classes 
of hypercyclic operators $T$,  $\hct\cup\{0\}$ contains a closed infinite-dimensional subspace. See for instance \cite{GLM00} for a characterization in spectral terms of operators acting on complex separable 
Banach spaces with this property. We mention also the work \cite{Men}.
\par\smallskip
When $X$ is a topological algebra, it makes sense to ask whether $\hct\cup\{0\}$ contains a non-trivial subalgebra of $X$. Such an algebra will be called a \emph{hypercyclic algebra}, and whenever it 
 exists, we will say that $T$ \emph{admits a hypercyclic algebra}. This question of the existence of a hypercyclic algebra was first considered by Bayart and Matheron in \cite{BayMat}*{Chapter 8}, and 
independently by Shkarin in \cite{Shk}: they showed that the differentiation operator $D \colon f\mapsto f'$  on the algebra $\oc$ of 
entire functions on the complex plane admits a hypercyclic algebra. On the other hand, the translation operators $T_{a} \colon f\mapsto f(\,\cdot+a)$, $a\neq 0$, acting on $\oc$ do not admit hypercyclic algebras \cite{ACPS}. The study of the existence of hypercyclic algebras has by now developed into a flourishing branch of linear dynamics. See for instance the works \cite{BesCP}, \cite{Bay1},   \cite{BC19}, \cite{BP20}, \cite{BEP20}, \cite{FGE20a} and \cite{BCP} among many other relevant 
references. 
\par\smallskip
Several of these papers deal with the important question of characterizing the entire functions $\phi $ of exponential type such that $\phi (D)$ admits a hypercyclic algebra. The paper 
\cite{Bay1} by Bayart introduces an approach to this problem based on the study of the eigenvectors of the operator $\phi (D)$, which will be of particular importance in this paper.
\par\smallskip 
Our aim is to study the question of the existence of hypercyclic algebras from the Baire Category point of view. Recall that if $(E,\tau )$ is a Polish space (i.e.\ separable and 
completely metrizable), and (P) is a certain property of elements of $E$, we say that (P) \emph{is typical} (or equivalently, that \emph{a typical} $x\in E$ \emph{has property} (P)) if 
the set $\{x\in E\;;\;x\ \textrm{has (P)}\}$ is comeager in $E$. A comeager set  in $E$ is a set which contains a dense $G_{\delta }$ set, i.e.\ which is large in $E$ in the sense of Baire Category. 
Thus given a particular Polish space $(E,\tau )$, it may  be of interest to determine whether some natural properties of elements of 
$E$ are typical or not.
\par\smallskip 
Given a separable Fr\'echet algebra $X$ (i.e.\ a separable completely metrizable topological algebra), our general goal is to define some natural spaces $\mathcal{L}$ of continuous linear 
operators on $X$, which are Polish spaces when endowed with the topology of pointwise convergence on $X$, and to determine whether a typical operator $T$ in such a space $\mathcal{L}$ possesses a 
hypercyclic algebra. When $(X, \norm{\,\cdot\,})$ is a separable Banach algebra, natural spaces $\mathcal{L}$ to consider are the closed balls $\bmx$ consisting of bounded operators $T$ on $X$ 
with $\norm{T}\le M$, with $M>1$ (so as to have a chance that a typical $T\in\bmx $ for the topology of pointwise convergence on $X$ is hypercyclic). In this Banach space setting, we denote by $\bx$ the 
algebra of bounded operators on $X$. The topology on $\bx$ defined as the topology of pointwise convergence on 
$X$ is usually called the \emph{Strong Operator Topology} (\sot): if $T_{\alpha }$ is a net of elements of $\bx$, and if $T\in\bx$,
% $T_\alpha \xrightarrow[\, \alpha \,]{}T$
$T_\alpha \xrightarrow {}T$ %
% $\xymatrix{T_{\alpha }\ar[r]_{\alpha }&T}$
for the \sot\ if and only if 
$T_\alpha x \xrightarrow {\norm{\,\cdot\,}} Tx$ 
% $\xymatrix{T_{\alpha }x\ar[r]_{\alpha }^{\|\cdot\|}&Tx}$
for every $x\in X$. When $X$ is a separable Banach space, $(\bmx,\sot)$ is a Polish space for every $M>0$. The study of typical 
properties of operators $T\in(\bmx,\sot)$  was initiated in the Hilbertian setting by Eisner and M\'{a}trai~\cite{EisMat} and continued in the works \cite{GMM2} and \cite{GMM3} in the case 
where $X=\ell_{p}(\nn)$, $1\le p<+\infty$, or $X=c_{0}(\nn)$. Typical properties of operators for other Polish topologies on closed balls $\bmx$ were also studied in \cite{Eis} and \cite{GriMat}, as 
well as in the monograph \cite{GMM1}.
\par\smallskip
When $X$ is a separable Fr\'echet space, we denote by $\mathcal{L}(X)$ the algebra of continuous linear operators on $X$, and  by \sot\ the topology on $\mathcal{L}(X)$ of pointwise convergence on $X$. Let $(N_{j})_{j\geq1}$ be a sequence of semi-norms on $X$ defining its topology. Given a sequence $(M_{j})_{j\geq1}$ of positive real numbers and a sequence  $(k_{j})_{j\geq1}$ of positive integers, we 
define
\[
\clmk(X)\coloneqq\bigl\{T\in\mathcal{L}(X)\;;\;\forall j\geq1,\ \forall x\in X,\ N_{j}(Tx)\le M_{j}\,N_{k_{j}}(x)\bigr\}\cdot
\]
We will show in Proposition \ref{propositionpolonais} below that $(\clmk(X),\sot)$ is a Polish space. Of course, this space may be very small for certain choices of the sequences 
$(M_{j})_{j\geq1}$ and $(k_{j})_{j\geq1}$, but in general, it makes sense to study \sot-typical properties of elements of $\clmk(X)$. Observe that the space 
$\clmk(X)$ depends on the choice of the sequence of semi-norms $(N_{j})_{j\geq1}$.
\par\smallskip
The main result of this paper deals with the case where $X=\oc$, the space of entire functions on $\cc$ endowed with the topology of uniform convergence on compact sets. This topology on $\oc$ can be defined via the following sequence $(N_{j})_{j\geq1}$ of semi-norms:
\[
\textrm{for every } f\in \oc,\ f(z)=\sum_{m\geq0}a_{m}z^ {m},\ z\in\cc,\quad \textrm{let}\
N_{j}(f)\coloneqq\sum_{m\geq0}|a_{m}|\,j^ {m}.
\]
There are of course plenty of different choices of sequences of semi-norms defining the topology of $\oc$, but for our purposes this choice is the most convenient. To simplify the notation, 
we will write $\clmk\coloneqq\clmk(\oc)$, subordinated to this choice of the sequence $(N_{j})_{j\geq1}$.
\par\medskip
We are now ready to state our main result.
\begin{theorem}\label{maintheorem}
 Let the sequences $(M_{j})_{j\geq1}$ and $(k_{j})_{j\geq1}$ be such that 
 \begin{enumerate}[\normalfont (i)]
  \item for all $\alpha\ge 1$,  $j^\alpha=o(M_j)$ as $j$ tends to infinity;
  \item for all $j\geq 1$, $M_j\geq j+1$;
  \item for all $j\geq 1$, $k_j\geq j+2$;
  \item $M_1<k_1$ and  $k_j\geq k_1$ for all $ j\ge 1$.
 \end{enumerate}
Then a typical operator $T\in(\clmk,\emph{\sot})$ admits a hypercyclic algebra.
\end{theorem}
Choosing the sequences  $(M_{j})_{j\geq1}$ and $(k_{j})_{j\geq1}$ in such a way that for all $j\geq 1$, $M_j\geq j+1$ and $k_j\geq j+2$ ensures that the differentiation operator $D$ -- the crucial example of an operator  on 
$\oc$ which admits a hypercyclic algebra -- belongs to $\clmk$. The operators $D$ and $\frac{1}{n!} D^n$ will be of special importance in the proof of Theorem \ref{maintheorem}; they will play a role analogous to the left shift operator and its iterates in the Banach space case.
\par\smallskip
Before going further, we notice that the differentiation operator is a so-called \emph{tame operator}, and that the spaces $\clmk$ in Theorem \ref{maintheorem} include spaces of tame operators. They are thus natural spaces to be considered in the context of Fr\'echet spaces. We recall briefly the definition of a tame operator, and we refer to the survey \cite{Ham} for more on tame Fr\'echet spaces. Let $\tilde X$ be a Fr\'echet space whose topology is given by a certain increasing family of semi-norms $(\tilde N_j)_{j\ge 1}$, and let $\tilde T\colon \tilde X\to \tilde X$ be a linear map.
  The map $\tilde T$ is called \emph{tame} if there exist two integers $b,r\ge 1$ such that for every  $j\ge b$, the following is true:
there exists a constant $C_j>0$ such that $\tilde N_j(\tilde Tx)\le C_j \tilde N_{j+r}(x)$ for every $x\in\tilde X$. A tame linear map is automatically continuous.
It can easily be observed that the differentiation operator $D$ is tame, and the operators belonging to the space $\clmk$ considered in Theorem \ref{maintheorem} are tame provided  that there  exists an integer $r\geq 2$ such that $k_j=j+r$ for all $j$ sufficiently large.

\par\smallskip

In order to prove Theorem \ref{maintheorem}, we must exhibit a dense set of operators $T$ belonging to $\clmk$ admitting a hypercyclic 
algebra. These operators will be of the form
\[
T=A+\dfrac{1}{(n+1)!}D^ {n+1}
\]
for some integer $n\ge 0$,
where $A$ acts on $\cc_{n}[z]$ (the vector space of complex polynomials of degree at most $n$) as a \emph{nilpotent} endomorphism, and $A(z^{j})=0$ for every $j>n$. We note that these operators are in general not of the form $\phi (D)$, where $\phi $ is an entire function of 
exponential type (compare to \cite{Bay1}, \cite{BCP}, \cite{BesCP}); thus they provide {a new family} of operators acting on $\oc$ that admit hypercyclic algebras.
\par\medskip 
We also investigate the case of the complex Fr\'echet algebras $X=\ell_{p}(\nn)$, $1\le p<+\infty$ or $X=c_{0}(\nn)$ endowed with the coordinatewise  product: if $x=(x_{n})_{n\geq 1}$ and 
$y=(y_{n})_{n\geq 1}$ are two sequences of complex numbers belonging to $X$, we define $x \cdot y=(x_{n}y_{n})_{n\geq 1}$. Then $\norm{x \cdot y}\le \norm{x}\,\norm{y}$, where 
$\norm{\,\cdot\,}$ denotes the classical $\ell_{p}$- or $c_{0}$-norm on $X$. As $X$ is a Banach space, we place ourselves in the closed balls $\bmx$ of $\bx$, $M>1$, and we prove the following theorem.
\begin{theorem}\label{Theorem 2}
 Let $X=\ell_{p}(\nn)$, $1\le p<+\infty$, or $X=c_{0}(\nn)$, endowed with the coordinatewise product. Let $M>1$. A typical operator $T\in(\bmx,\emph{\sot})$ admits a hypercyclic algebra.
\end{theorem}
The proofs of Theorems \ref{maintheorem} and \ref{Theorem 2} build on the approach via eigenvectors introduced by Bayart in \cite{Bay1}, and further developed by Bayart et al.\ in 
\cite{BCP}. These results show, in two different classical contexts, that having a hypercyclic algebra is a quite common phenomenon, at least from the Baire Category point of view. 

However, at this point a word of caution is in order: as explained in \cite{GMM2}*{Proposition 3.2}, hitherto the properties (P) of operators (on $X=\ell_{p}(\nn)$ or $c_{0}(\nn)$) which have  been studied from the typicality perspective are either typical or ``atypical'' in $(\bmx,\sot)$; in other words, for such a property (P), either a typical $T\in(\bmx,\sot)$ has (P), or a typical  $T\in(\bmx,\sot)$ does not possess (P). This relies on the topological $0-1$ law (cf.\ \cite{Kec}*{Theorem 8.46}, and \cite{GMM2}*{Proposition 3.2} for its application in our context) and on the fact that for the properties under consideration, $T$ has (P) if and only if 
$JTJ^{-1}$ has (P) for every surjective isometry $J$ of $X$. In our setting, this argument fails: having a hypercyclic algebra is not a property  which is a priori  stable by conjugation by 
invertible isometries (cf.\ \cite{BCP}*{Remark\ 4.7}, where it is observed that admitting a hypercyclic algebra is not a property which is preserved by similarity).
\par\medskip
The paper is organized as follows: Sections \ref{section1bis}, \ref{section2}, \ref{section3} and \ref{section4} are devoted to the proof of our main result, 
Theorem \ref{maintheorem}.
We introduce in Section \ref{section1bis} 
 our Polish spaces of operators on the algebra of entire functions; then a technical result regarding the density of certain classes of operators in $(\clmk,{\sot})$  is proved in Section \ref{section2}; in Section \ref{section3} we show, using the Godefroy-Shapiro Criterion, that hypercyclic operators form a dense $G_\delta$ subset of $(\clmk,{\sot})$; finally in Section \ref{section4}, after  recalling the Baire Category Criterion from 
\cite{BCP}*{Corollary 2.4} which we use  to prove the existence of a hypercyclic algebra for the operators under consideration, we prove Theorem \ref{maintheorem}. Theorem \ref{Theorem 2} is proved in 
Section \ref{section5}. Section \ref{section6} contains the additional result that when $X=\ell_{p}(\nn)$, $1\le p<+\infty$, or $X=c_{0}(\nn)$, and $M>1$, a typical 
$T\in(\bmx,\sot)$ admits a closed infinite-dimensional hypercyclic subspace (Theorem \ref{theorem:HcSubspTypicalSOT}).

\section{Polish Space of operators on the algebra of entire functions.}\label{section1bis}

We present in this section some general facts which will be of use in the sequel. We begin by proving the following proposition for a general separable Fr\'echet space $X$.
\begin{proposition}\label{propositionpolonais}
Let $X$ be a separable Fr\'echet space, and let $(N_j)_{j\geq 1}$ be a sequence of semi-norms defining its topology. Let  $(M_{j})_{j\geq1}$ be a sequence of positive real numbers and let $(k_{j})_{j\geq1}$ be a sequence of positive integers. Then
$$
\clmk(X)\coloneqq\bigl\{T\in\mathcal{L}(X)\;;\;\forall j\geq1,\ \forall x\in X,\ N_{j}(Tx)\le M_{j}\,N_{k_{j}}(x)\bigr\}
$$ is a Polish space when endowed with the $\emph{\sot}$.
\end{proposition}

\begin{proof} The proof is similar to that of \cite{GMM2}*{Lemma 3.1}. We recall it here briefly for the sake of completeness. We denote by $Z$ a countable dense subset of $X$. Then let
$$\tilde{\mathcal{L}}_{(M_{j}),(k_{j})}(Z)\coloneqq\{T\colon Z\to X;\ T\text{ is linear and } \forall j\ge 1, \forall x\in Z, \ N_j(Tx)\leq M_j N_{k_j}(x)\}.$$ 
Then $\tilde{\mathcal{L}}_{(M_{j}),(k_{j})}(Z)$ is closed in $X^Z$, where $X^Z$ is endowed with the product topology. Indeed, since $X^Z$ is metrizable, let $(T_k)_k$ be a sequence of elements of $\tilde{\mathcal{L}}_{(M_{j}),(k_{j})}(Z)$ which converges to $T$. Since the product topology in $X^Z$ is the pointwise topology, $T$ is necessarily linear; since for all $j\geq 1$, and all $x\in Z$, the sequence $(N_j(T_kx))_k$ converges to $N_j(Tx)$, we also have that $N_j(Tx)\leq M_j N_{k_j}(x)$, and so $T$ belongs to $\tilde{\mathcal{L}}_{(M_{j}),(k_{j})}(Z)$.
\par\smallskip
Since $X^Z$ is a Polish space, $\tilde{\mathcal{L}}_{(M_{j}),(k_{j})}(Z)$ is a Polish space as a closed subset of a Polish space. Now, we observe that the map $\Phi\colon (\clmk(X), \sot)\to \tilde{\mathcal{L}}_{(M_{j}),(k_{j})}(Z)$ defined by $\Phi(T)=T|_Z$ is a homeomorphism, and we conclude that $(\clmk(X),\sot)$ is also a Polish space.\end{proof}

\par\medskip

From now on, we will consider the case where $X=H(\C)$, the space of entire functions on the complex plane endowed with the topology of uniform convergence on compact sets.
We denote by $(z^k)_k$ the monomial basis of $H(\C)$ and for $z\in\cc$ and $r>0$, we let $D(z,r)$ be the open disk of center $z$ and  radius $r$. We record here the following property.
\begin{property}\label{propri??t??1}
 For every $f\in \oc$ and every $j\ge 1$, we have
$$\sup_{D(0,j)}|f|\leq N_j(f)\leq (j+1)\sup_{D(0,j+1)}|f|.$$
\end{property}
\begin{proof} The lefthand-side inequality is trivial, and the righthand-side inequality follows from Cauchy's inequalities: for all $k\in\nn$ and all $r>0$, we have 
\begin{align*}
 \frac1{k!}\left|\diffp{^kf}{z^k}(0)\right|&= \frac1{2\pi} \left|\int_{|\zeta|=r} \frac{f(\zeta)}{\zeta^{k+1}} d\zeta\right|
 \leq \frac1{r^k}\sup_{D(0,r)}|f|.
\end{align*}
This implies that
\begin{align*}
 N_j(f)
 &\le \sum_{k=0}^{+\infty} \left(\frac{j}{j+1}\right)^k\sup_{D(0,j+1)}|f|=(j+1)\sup_{D(0,j+1)}|f|.
\end{align*}
\end{proof}
 \par\smallskip

Therefore a sequence $(f_n)_n$ of functions in $\oc$ converges to $f$ in $\oc$ if and only if the sequence $(N_j(f_n-f))_n$ converges to $0$ for every $j\ge 1$; thus the family $(N_j)_j$ indeed induces the topology of $\oc$.

 \par\smallskip

The next property provides a useful neighborhood basis of an element of $\clmk$ for the ${\sot}$.

\begin{property}\label{propri??t??2}
Let $T_0$ belong to $\clmk$.
 A neighborhood basis of $T_0$ for the \emph{\sot} is given by the family of sets
$$\cv_{\varepsilon,r,K}^{T_0}=\{T\in \clmk\, ;\, \forall k=0,\ldots, K, \ N_r((T-T_0)z^k)<\varepsilon\},$$ 
where $\varepsilon>0$, and $K,r$ are integers with $K\ge 0$ and $r\ge 1$.
\end{property}
\begin{proof} It suffices to show that if $(T_k)_k$ is a sequence of elements of $\clmk$ such that, for all $\varepsilon>0$, all $K\ge 0$ and all $r\ge 1$, there exists $\kappa_0\in\nn$ such that for all $k\geq \kappa_0$, $T_k$ belongs to $\cv_{\varepsilon,r,K}^0$, then $(T_k)_k$ converges \sot \ to $0$.
\par\smallskip
Let $f$ be an entire function and let $j\ge 1$ be a positive integer. There exists an analytic nonzero polynomial $P$ such that $N_{k_j}(f-P)<\frac{\varepsilon}{2M_j}$. Denote by $d$ the degree of $P$, write $P$ as $P(z)=\sum_{l=0}^d p_lz^l$, and set $\varepsilon'=\frac{\varepsilon}{2\sum_{l=0}^d |p_l|}$. By assumption, there exists $\kappa_0\in\nn$ such that for all $k\geq \kappa_0$, $T_k$ belongs to $\cv_{\varepsilon',d,j}^0$. Therefore we have 
\begin{align*}
 N_j(T_kf)
 &\leq N_j(T_k(f-P))+N_j(T_k P)
 \leq M_j N_{k_j}(f-P)+\sum_{l=0}^d |p_l|N_j(T_kz^l)\\
 &\leq M_j \frac{\varepsilon}{2M_j} + \varepsilon'\sum_{l=0}^d |p_l|=\varepsilon.
\end{align*}
This implies that the sequence $(T_kf)_k$ converges to $0$ uniformly on any compact subset of $\cc$, and thus $(T_k)_k$ converges \sot\ to 0. \end{proof} 
\par\medskip
In the forthcoming proof of Theorem \ref{maintheorem}, the operators
 $S_n\coloneqq\frac1{n!} D^n$, $n\geq 1$, will play a crucial role. It is thus important to be able to demonstrate that they belong to the space $\clmk$. In the next proposition, we provide conditions on the sequences $(M_j)_{j\geq1} $ and $(k_j)_{j\geq 1}$ which ensure that this is indeed the case.
 
\begin{proposition}\label{proposition3}
 Let $(M_j)_{j\ge 1}$ be a sequence of positive real numbers and let $(k_j)_{j\ge 1}$ be a sequence of positive integers such that for every $j\ge 1$, $M_j\geq j+1$ and $k_j\geq j+2$. Then $S_n$ belongs to $\clmk$ for every $n\ge 1$.
\end{proposition}
\begin{proof} Let $n\ge 1$, $j\ge 1$, and let $z\in D(0,j+1)$. By Cauchy's formula, we have:
$$\frac1{n!}\diffp{^nf}{z^n}(z)=\frac1{2i\pi}\int_{|\zeta-z|=1} \frac{f(\zeta)}{(\zeta-z)^{n+1}}d\zeta$$ so that $|S_n(f)(z)|\leq \sup_{D(0,j+2)}|f|$. We then deduce from Property \ref{propri??t??1} that 
\begin{align*}
N_j(S_nf) 
&\leq (j+1) \sup_{D(0,j+1)}|S_n(f)|
\leq (j+1) \sup_{D(0,j+2)}|f|
\leq (j+1) N_{j+2}(f),
\end{align*}  
and this proves that $S_n$ belongs to $\clmk$ when $(M_j)_{j\ge 1}$ and $(k_j)_{j\ge 1}$ satisfy the assumptions of this proposition.
\end{proof}

\section{Dense families of operators in $\clmk$}\label{section2}

For every $n\ge 0$, we define the operator $T_n\colon \oc\to\oc$ by setting $$T_nf(z)=\sum_{j=0}^n\frac1{j!} \diffp{^jf}{z^j}(0)z^j,\quad \textrm{where } f\in\oc$$ and we note that for all $j\ge 1$, $N_j(T_nf)\leq N_j(f)$. 
\par\smallskip
Given $A\in\clmk$, our aim is to approximate $A$ with an operator supporting a hypercyclic algebra. We first approximate $A$ by an operator of the form $B_n=T_n A T_n$. When $n$ goes to infinity, $(B_n)_n$ converges \sot\ to $A$, and for all $n\ge 1$, $B_n$ belongs to $\clmk$. Indeed, we have for every $f\in\oc$ and every $n,j\ge 1$ that
\begin{align*}
 N_j(B_nf)
 \leq N_j(A T_n f)
 \leq M_jN_{k_j}(T_nf)
 \leq M_jN_{k_j}(f).
 \end{align*}
 \par\smallskip
When $M_j\geq j+1$ and $k_j\geq j+2$ for all $j\ge 1$, it follows from Proposition \ref{proposition3} that $S_{n+1}$ belongs to $\clmk$ {for every $n \geq 0$}. Hence, for every $\delta\in[0,1]$, the operator $(1-\delta)B_n+\delta S_{n+1}$ also belongs to $\clmk$. Moreover, $\left((1-\delta)B_n+\delta S_{n+1}\right)_{\delta\in (0,1)}$ \sot\ converges to $B_n$ as $\delta$ goes to $0$. Thus we have proved the following.

\begin{proposition}\label{approximationhypercyclique}
 Let $(M_j)_{j\ge 1}$ be a sequence of positive real numbers, and let $(k_j)_{j\ge 1}$ be a sequence of positive integers such that  $M_j\geq j+1$ and $k_j\geq j+2$ for every $j\ge 1$. Then the following holds: 
 for every $A\in \clmk$ and every \emph{\sot}-neighborhood $\cv$ of $A$, there exist $n\ge 0$, an operator $B\in\clmk$ satisfying $B=T_n B T_n$,  and $\delta\in(0,1)$ such that $B+\delta S_{n+1}$ belongs to $\cv$. 
\end{proposition}

We will prove in Section \ref{section3} below that operators of the form $B+\delta S_{n+1}$ are hypercyclic, where  $\delta>0$, $n\ge 0$ and $B$ is such that $B=T_n B T_n$. So the reader mainly interested in the density of hypercyclic operators in $\clmk$ can proceed directly to Section \ref{section3}.
\par\smallskip
On the other hand, it does not seem to be trivial that an operator of the form $B+\delta S_{n+1}$ with  $\delta\in(0,1)$, $n\ge 0$ and $B$ such that $B=T_n B T_n$ supports a hypercyclic algebra. What we will be able to prove in Section \ref{section4} is that if for every polynomial $P$, the sequence $(B^j P)_j$ converges to $0$ in $\oc$, then $B+\delta S_{n+1}$ indeed supports a hypercyclic algebra (cf.\ Theorem \ref{operateur_algebre_hypercyclique}). With this result in view, we now show that the family of nilpotent operators $B$ satisfying $B=T_n B T_n$ for some $n\ge 0$ is \sot\ dense in $\clmk$ provided that the sequences $(M_j)_j$ and $(k_j)_j$ satisfy some suitable assumptions, which in fact ensure that $\clmk$ contains sufficiently many operators.

\begin{proposition}\label{approximationalgebrehypercyclique}
 Let $(M_j)_{j\ge 1}$ be a sequence of positive real numbers, and let $(k_j)_{j\ge 1}$ be a sequence of positive integers such that 
 \begin{enumerate}[\normalfont (i)]
  \item\label{condi} for all $\alpha\ge 1$, $j^\alpha=o({M_j})$ as $j$ tends to infinity;
  \item\label{condii} $k_j>j$ for every $j\ge 1$;
  \item\label{condiii} $M_1<k_1$ and $k_j\geq k_1$ for every $j\ge 1$.
 \end{enumerate}
Let $A$ be an operator belonging to $\clmk$ and let $\cv$ be an \emph{\sot}-neighborhood of $A$ in $\clmk$. Then there exist $B\in \cv$ and $n\ge 1$ such that $B=T_n B T_n$ (in particular, for all $f\in\oc$, $Bf$ is a polynomial of degree at most $n$) and $B^{n+1} =0$.
\end{proposition}

\begin{remark}
 Conditions \ref{condi} and \ref{condii} of Proposition \ref{approximationalgebrehypercyclique} ensure that $\clmk$ contains sufficiently many operators. For example, for every $n\ge 0$ and every $a\in\cc$ close enough to $0$ (how close depends on $n$), the operator $B_n$ defined  by $B_nf=az^nf$, $f\in\oc$, belongs to $\clmk$. Indeed, for every $j\ge 1$ and every $i\ge 1$ we have 
 \begin{align*}
  N_j(B_n z^i)&=N_j(az^{i+n})
  = |a| j^{i+n}.
 \end{align*}
Thus $B_n$ belongs to $\clmk$ if and only if  $|a| j^{i+n}\leq M_j k_j^i$ for every $j\ge 1$. Condition \ref{condii} ensures in particular that $j\leq k_j$, so that $ N_j(B_n z^i)\le |a| j^{n} k_j^i$ for every $i,j\ge 1$. Condition \ref{condi} ensures the existence of $j_0\ge 1$ such that $j^n\leq M_j$ for every $j\ge j_0$; if $|a|\le 1$, 
$ N_j(B_n z^i)\le M_j k_j^i$ for every $i\ge 1$ and every $j\ge j_0$.
Now, if $|a|$ is small enough, depending on $n$, we can ensure that $|a|j^n \leq M_j$ for every $1\le j\leq j_0$ and thus $N_j(B_nz^i)\leq M_j k_j^i$ for every $i\ge 1$ and every $j\le j_0$. Hence $B_n$ belongs to $\clmk$.

Condition \ref{condiii} is a technical condition which will be needed in the proof of Proposition \ref{approximationalgebrehypercyclique}. It does not appear to be overly restrictive.
\end{remark}

\begin{remark}
In the Hilbertian setting, it is known that \emph{nilpotent} contractions on a complex separable Hilbert space $H$ are \sotstar-dense in $\mathcal{B}_1(H)$, so in particular \sot-dense (\cite{GriMat}*{Proposition 4.6}). It seems to be unknown whether this result can be extended to $\ell_p$-spaces, $1<p<+\infty$, $p \neq 2$. 
\end{remark}

The rest of this section is devoted to the proof of Proposition \ref{approximationalgebrehypercyclique}.

\begin{proof}[Proof of Proposition \ref{approximationalgebrehypercyclique}]
Without loss of generality (cf.\ Property \ref{propri??t??2}), we assume that $\cv$ has the form
$$\cv=\cv_{\varepsilon_0,r,K}^A=\{B\in\clmk\,;\, \ N_r(Bz^j-Az^j)<\varepsilon_0, \ j=0,\ldots, K\},$$
where $\varepsilon_0>0$, and $r, K\geq 1$ are fixed.

Let $n_0\geq K$ be a sufficiently large integer and $\delta\in(0,1)$ a sufficiently small positive number. We set 
$$A_0=(1-\delta)T_{n_0} A T_{n_0}.$$
If $n_0$ is sufficiently large and $\delta$ is sufficiently small, then $A_0$ belongs to $\cv$. We now fix the parameters $n_0\geq K$ and $\delta\in(0,1)$.
We then consider an operator $B$ defined as follows:
\begin{align*}
 Bz^j=
\begin{cases}
 A_0z^j+\varepsilon z^{n_0+1+j}, & 0\le j\le n_0\\
 \varepsilon z^{n_0+1+j}, & n_0+1\le j\le (n_0+1)(n_1+1)-1\\
 z^{n_0+1+j}, & (n_0+1)(n_1+1) \le j\le  (n_0+1)n_2-1\\
{D_{n_0,n_1,n_2}(z^j)}, & (n_0+1)n_2 \le j\le (n_0+1)(n_2+1)-1\\
 0, & j\geq (n_0+1)(n_2+1)
\end{cases}
\end{align*}
 where 
 \begin{align*}
D_{n_0,n_1,n_2}(z^j) \coloneqq	-\sum_{m=0}^{n_1} & \frac1{\varepsilon^{n_1-m+1}} z^{(n_0+1)m} A_0^{n_2+1-m} z^{j-n_2(n_0+1)} \\ &-\sum_{m=n_1+1}^{n_2}  z^{(n_0+1)m} A_0^{n_2+1-m} z^{j-n_2(n_0+1)}
\end{align*}
and where  $n_1,n_2 \ge 1$, $\varepsilon>0$ have to be determined in order to ensure that $B$ belongs to $\clmk$. 
Our first goal is to show that with this definition, $B$ is nilpotent -- more precisely that $B^{n_2+1}=0$.

\begin{fact}\label{fait1}
 For every $P\in\cc_{n_0} [z]$ and every $1\leq n\leq n_2$, we have
 \begin{align}
B^nP&=\sum_{k=0}^{\min(n,n_1)}\varepsilon^k z^{k(n_0+1)} A_0^{n-k} P +\varepsilon^{n_1+1} \sum_{k=n_1+1}^n z^{k(n_0+1)} A_0^{n-k} P. \label{eq1}
\end{align}
\end{fact}
\begin{proof} Equality (\ref{eq1}) is true for $n=1$, since by definition $BP=A_0P+\varepsilon z^{n_0+1}P$. If we assume that (\ref{eq1}) is true for some $n\ge 1$, then, for every $1\leq k\leq n_1$, $z^{k(n_0+1)} A_0^{n-k} P$ is a sum of monomials whose degrees lie between $k(n_0+1)$ and $(k+1)(n_0+1)-1$; since $k(n_0+1)\geq n_0+1$ and $(k+1)(n_0+1)-1\leq (n_1+1)(n_0+1)-1$, it follows that
\begin{align*}
 B(z^{k(n_0+1)} A_0^{n-k}P)&=\varepsilon z^{(k+1)(n_0+1)} A_0^{n-k} P \text{ for } 1\leq k\leq n_1.
\end{align*}
Analogously, for $n_1+1\leq k\leq n_2-1$, $z^{k(n_0+1)}A_0^{n-k}P$ is a sum of monomials whose degrees lie between $k(n_0+1)$ and $(k+1)(n_0+1)-1$, with $k(n_0+1)\geq (n_1+1)(n_0+1)$ and $(k+1)(n_0+1)\leq n_2(n_0+1)-1$. Thus
\begin{align*}
 B(z^{k(n_0+1)} A_0^{n-k}P)&=z^{(k+1)(n_0+1)}A_0^{n-k}P\quad \text{ for every } n_1\leq k\leq n_2-1.
\end{align*}
We now consider separately two cases.
\par\smallskip
\emph{Case 1:} If $n<n_1$, we have
\begin{align*}
B^{n+1}P
&=BA_0^nP+\sum_{k=1}^n \varepsilon^k B\left(z^{k(n_0+1)} A_0^{n-k} P\right)\\
&= A_0^{n+1}P +\varepsilon z^{n_0+1} A_0^n P+\sum_{k=1}^n \varepsilon^{k+1} z^{(k+1)(n_0+1)}A_0^{n-k} P\\
&=  A_0^{n+1}P +\varepsilon z^{n_0+1} A_0^n P+\sum_{k=2}^{n+1} \varepsilon^{k} z^{k(n_0+1)}A_0^{n+1-k} P\\
&=\sum_{k=0}^{\min(n+1,n_1)}\varepsilon^k z^{k(n_0+1)} A_0^{n+1-k} P.
\end{align*}
Hence (\ref{eq1}) still holds  true at rank $n+1$ when $n<n_1$.
\par\smallskip
\emph{Case 2:} 
If $n_1\leq n<n_2$, we have
\begin{align*}
B^{n+1}P&= BA_0^n P + \sum_{k=1}^{n_1} \varepsilon^k B\left(z^{k(n_0+1)}A_0^{n-k} P\right)
 +\varepsilon^{n_1+1}\sum_{k=n_1+1}^n B\left(z^{k(n_0+1)}A_0^{n-k} P\right)\\
&=A_0^{n+1}P+\varepsilon z^{n_0+1} A_0^nP 
  +\sum_{k=1}^{n_1} \varepsilon^{k+1} z^{(k+1)(n_0+1)}A_0^{n-k} P\\
  &\qquad +\varepsilon^{n_1+1} \sum_{k=n_1+1}^{n}  z^{(k+1)(n_0+1)}A_0^{n-k} P.
\end{align*}
This gives that
\begin{align*}
B^{n+1}P&=A_0^{n+1}P+\varepsilon z^{n_0+1} A_0^nP 
  +\sum_{k=2}^{n_1+1} \varepsilon^{k} z^{k(n_0+1)}A_0^{n+1-k} P\\
 & \qquad +\varepsilon^{n_1+1} \sum_{k=n_1+2}^{n+1}  z^{k(n_0+1)}A_0^{n+1-k} P\\
&=A_0^{n+1}P+\varepsilon z^{n_0+1} A_0^nP 
  +\sum_{k=2}^{n_1} \varepsilon^{k} z^{k(n_0+1)}A_0^{n+1-k} P\\
  &\qquad +\varepsilon^{n_1+1} \sum_{k=n_1+1}^{n+1}  z^{k(n_0+1)}A_0^{n+1-k} P.
\end{align*}
Hence (\ref{eq1}) is still true at rank $n+1$ in this second case where $n_2>n\geq n_1$, which ends the proof of Fact \ref{fait1}.
\end{proof}
\par\smallskip
\begin{fact}\label{fait2}
 For every $P\in \cc_{n_0}[z]$, we have $B^{n_2+1}P=0$.
\end{fact}

\begin{proof} Fact \ref{fait1} applied to $n=n_2$ yields that
\begin{align*}
B^{n_2+1}P&=B\left(\sum_{k=0}^{n_1}\varepsilon^k z^{k(n_0+1)} A_0^{n_2-k} P +\varepsilon^{n_1+1} \sum_{k=n_1+1}^{n_2} z^{k(n_0+1)} A_0^{n_2-k} P\right)\\
&= A_0^{n_2+1} P +\varepsilon z^{n_0+1} A_0^{n_2}P+\sum_{k=1}^{n_1} \varepsilon^{k+1} z^{(k+1)(n_0+1)} A_0^{n_2-k} P\\
&\qquad +\varepsilon^{n_1+1} \sum_{k=n_1+1}^{n_2-1} z^{(k+1)(n_0+1)}A_0^{n_2-k}P
-\varepsilon^{n_1+1} \sum_{m=0}^{n_1}\frac1{\varepsilon^{n_1-m+1}} z^{m(n_0+1)}A_0^{n_2+1-m}P\\
&\qquad -\varepsilon^{n_1+1} \sum_{m=n_1+1}^{n_2} z^{m(n_0+1)}A_0^{n_2+1-m}P.
\end{align*}
Writing the sum
$$\varepsilon^{n_1+1} \sum_{m=0}^{n_1}\frac1{\varepsilon^{n_1-m+1}} z^{m(n_0+1)}A_0^{n_2+1-m}P$$
as 
$$A_0^{n_2+1}P+\sum_{m=0}^{n_1-1}\varepsilon^{m+1} z^{(m+1)(n_0+1)}A_0^{n_2-m}P,$$
and the sum 
$$\varepsilon^{n_1+1} \sum_{m=n_1+1}^{n_2} z^{m(n_0+1)}A_0^{n_2+1-m}P$$
as
$$\varepsilon^{n_1+1} z^{(n_1+1)(n_0+1)}A_0^{n_2-n_1}P+\varepsilon^{n_1+1} \sum_{m=n_1+1}^{n_2-1} z^{(m+1)(n_0+1)}A_0^{n_2-m}P,$$
we obtain that
$B^{n_2+1}P=0$,
which proves Fact \ref{fait2}.
\end{proof}
\par\smallskip
We are now ready to prove the next fact.

\begin{fact}\label{fait3}
We have $B^{n_2+1}=0$.
\end{fact}

\begin{proof} Let us first check that the family $(1,\ldots, z^{n_0}, B1,\ldots, Bz^{n_0},\ldots, B^{n_2}1,\ldots, B^{n_2}z^{n_0})$ is a basis of $\cc_{(n_0+1)(n_2+1)-1}[z]$. Suppose that $\sum_{k=0}^{n_2} \sum_{j=0}^{n_0} \lambda_{k,j} B^kz^j=0$ for some complex numbers $\lambda_{k,j}$. Applying $B^{n_2}$ to this equality, we obtain by Fact \ref{fait2} that $\sum_{j=0}^{n_0} \lambda_{0,j} B^{n_2}z^j=0$.  Fact \ref{fait1} then implies that
\begin{align}
 \sum_{j=0}^{n_0} \lambda_{0,j} \left( 
 \sum_{k=0}^{n_1}\varepsilon^k z^{k(n_0+1)} A_0^{n_2-k} z^j +\varepsilon^{n_1+1} \sum_{k=n_1+1}^{n_2} z^{k(n_0+1)} A_0^{n_2-k} z^j
 \right) &=0\label{eq2}
\end{align}
so that the two polynomials
\begin{align*}
 \sum_{j=0}^{n_0} \lambda_{0,j} \left( 
 \sum_{k=0}^{n_1}\varepsilon^k z^{k(n_0+1)} A_0^{n_2-k} z^j +\varepsilon^{n_1+1} \sum_{k=n_1+1}^{n_2-1} z^{k(n_0+1)} A_0^{n_2-k} z^j
 \right)
 \end{align*}
 and
\begin{align*}
- \varepsilon^{n_1+1}\sum_{j=0}^{n_0} \lambda_{0,j}   z^{n_2(n_0+1)}  z^j
\end{align*}
are equal.
For $0\leq k\leq n_2-1$ and $0\le j\le n_0$, the degree of $z^{k(n_0+1)} A_0^{n_2-k}z^j$ is less than $n_2(n_0+1)-1$ while the degree of $z^{n_2(n_0+1)+j}$ is strictly greater than $n_2(n_0+1)-1$. It follows that $\sum_{j=0}^{n_0} \lambda_{0,j}z^{n_2(n_0+1)+j}=0$, and so that $\lambda_{0,0}=\cdots=\lambda_{0,n_0}$.
\par\smallskip
Suppose now that $\lambda_{0,0}=\cdots=\lambda_{0,n_0}=\cdots=\lambda_{l,0}=\cdots=\lambda_{l,n_0}$ for some $0\le l<n_2$. We know that $\sum_{k=l+1}^{n_2} \sum_{j=0}^{n_0} \lambda_{k,j} B^kz^j=0$, and we apply $B^{n_2-(l+1)}$ to this equality. We then obtain as previously that $\lambda_{l+1,0}=\cdots=\lambda_{l+1,n_0}$. By induction, $\lambda_{j,k}=0$ for all $j$ and $k$, and  $(1,\ldots, z^{n_0}, B1,\ldots, Bz^{n_0},\ldots, B^{n_2}1,\ldots, B^{n_2}z^{n_0})$ is indeed a basis of $\cc_{(n_0+1)(n_2+1)-1}[z]$.
\par\smallskip
Now, writing any polynomial $P\in\cc_{(n_0+1)(n_2+1)-1}[z]$ as $P=\sum_{k=0}^{n_2} \sum_{j=0}^{n_0} \mu_{k,j} B^kz^j$, we obtain $B^{n_2+1} P =0$. This proves Fact \ref{fait3}.
\end{proof}

The last step in our proof of Proposition \ref{approximationalgebrehypercyclique} is the following Fact \ref{fait4}, which shows that the operator $B$ belongs to $\cv$ for suitably chosen $n_1$, $n_2$ and $\varepsilon>0$.

\begin{fact}\label{fait4}
 There exist integers $n_1$ and $n_2$ large enough and $\varepsilon>0$ small enough such that the operator $B$ belongs to $\cv$. 
\end{fact}

\begin{proof}  Since $n_0\geq K$, we have $Bz^j=A_0 z^j+\varepsilon z^j$  for every $j=0,\ldots, K$, and thus
\begin{align*}
 N_r(Bz^j-Az^j)&\leq N_r(Bz^j-A_0z^j)+N_r(A_0z^j-Az^j)\\
 &\leq \varepsilon r^j + N_r(A_0z^j-Az^j).
\end{align*}
Since $A_0$ belongs to $\cv$, we have $ N_r(A_0z^j-Az^j)<\varepsilon_0$ for every $j=0,\ldots, K$. One can choose $\varepsilon>0$ small enough, so as to ensure that $N_r(Bz^j-Az^j)<\varepsilon_0$ also holds for every $j=0,\ldots, K$. Therefore, to demonstrate that $B$ belongs to $\cv$,  the main problem is to show that $B$ belongs to $\clmk$ for suitable choices of the integers $n_1$ and $n_2$ (which will have to be large) and of the real number $\varepsilon>0$ (which must be sufficiently small).
\par\smallskip
Since $\frac{M_j}{j^{n_0+1}} \rightarrow +\infty$ as $j\to +\infty$, there exists an integer $j_0\in\nn$ such that  $j^{n_0+1}\leq M_j$ for all $j\geq j_0$, and provided that $j_0$ is large enough, we also have $M_1 k_1^{n_0} \frac1{j_0^{n_0+1}}< 1$.
Once $j_0$ is fixed, we choose $n_1$, $n_2$ and $\varepsilon$ as follows:
\par\smallskip
\begin{enumerate}[(i)]
 \item \label{i} Since $k_j>j$ for all $j$, we can first choose an integer $n_1$ sufficiently large, so that  
 $$j_0^{n_0+1}\left(\frac{j}{k_j}\right)^{(n_1+1)(n_0+1)}\leq M_j\quad\textrm{ for every } 1\le j\le j_0.$$
 \item \label{ii} We next choose $\varepsilon>0$ satisfying the following two properties: $$0<\varepsilon\leq \delta\quad\textrm{ and }\quad\varepsilon j_0^{2n_0+1}\leq \delta \min_{1\leq j<j_0} M_j.$$
 \item \label{iii} Let $c_0\coloneqq\sum_{m=0}^{+\infty} \left(M_1k_1^{n_0}\frac1{j_0^{n_0+1}}\right)^m<+\infty$. We choose an integer $j_1>j_0$ such that $$c_0 \frac1{\varepsilon^{n_1+1}} M_1k_1^{n_0} \frac1{j_1}<1.$$
 \item \label{iv} Lastly, using  again the assumption that $k_j>j$ for all $j\ge 1$ and that $M_1<k_1$, we choose an integer $n_2$ sufficiently large, such that
 
 \begin{enumerate}[(a)]
  \item \label{iva} $c_0 \frac1{\varepsilon^{n_1+1}} M_1 k_1^{n_0} \left(\frac j{k_j}\right)^{n_2(n_0+1)}\frac1j<1$ for every $j_0\le j < j_1$;
  \item \label{ivb} $\frac{n_2+1}{\varepsilon^{n_1+1}} \left( \frac j{k_j}\right)^{n_2(n_0+1)} \left(\max_{1\le i\le  j_0} k_i\right)^{n_0} <1$ for every  $1\le j \le j_0$;
  \item \label{ivc} $\frac{n_2+1}{\varepsilon^{n_1+1}} k_1^{n_0}\left(\frac{M_1}{k_1}\right)^{n_2}<1$.
 \end{enumerate}
 
\end{enumerate}
\par\smallskip
For these choices of the parameters $n_1$, $n_2$ and $\varepsilon>0$, we will prove that $N_j(Bz^m)\leq M_j k_j^m$ for every integer $m$ with $0\leq m\leq (n_2+1)(n_0+1)-1$ and every $j\ge 1$.  Since $Bz^m=0$ for all $m\geq (n_2+1)(n_0+1)$, this will imply that $B$ belongs to $\clmk$. 

We fix $j\ge 1$, and separate the proof into four cases.
\par\medskip

\emph{Case 1}: $1\leq m \leq n_0$.
\par\smallskip
 We have $Bz^m=A_0 z^m+\varepsilon z^{n_0+1+m}$ so 
$N_j(Bz^m)\leq N_j(A_0 z^m)+\varepsilon N_j(z^{n_0+1+m})$. Notice that
 \begin{align*}
 N_j(A_0z^m)
 &=(1-\delta)N_j(T_{n_0} A T_{n_0}z^m)
 \leq (1-\delta)N_j(Az^m)
 \leq (1-\delta)M_j k_j^m
 \end{align*}
 from which we deduce that 
 \begin{align*}
  N_j(Bz^m)&\leq (1-\delta)M_j k_j^m + \varepsilon j^{n_0+1+m}.
 \end{align*}
 
If $j\geq j_0$, then  $j^{n_0+1}\leq M_j$ so $N_j(Bz^m)\leq (1-\delta+\varepsilon)M_j k_j^m$ and since $\delta\geq \varepsilon$, we get $N_j(Bz^m)\leq M_j k_j^m$.

If $1\le j<j_0$, $N_j(Bz^m)\leq (1-\delta)M_j k_j^m+\varepsilon j_0^{2n_0+1}$ and by \ref{ii} we also have that $N_j(Bz^m)\leq M_j k_j^m$.
\par\medskip
\emph{Case 2:}
$(n_1+1)(n_0+1)\leq m\leq n_2(n_0+1)-1$.
\par\smallskip
In this case we have $Bz^m=z^{m+n_0+1}$ so $N_j(Bz^m)=j^{m+n_0+1}$.

If $j\geq j_0$, we have  $j^{n_0+1}\leq M_j$ and $k_j>j$, so $N_j(Bz^m)\leq M_j k_j^m$.

If $1\le j<j_0$, 
\begin{align*}
N_j(Bz^m)&=j^{n_0+1}\left( \frac j{k_j}\right)^m k_j^m
\leq j_0^{n_0+1}\left( \frac j{k_j}\right)^{(n_1+1)(n_0+1)}  k_j^m
\leq M_j k_j^m
\end{align*}
where the last inequality follows from \ref{i}. 
\par\medskip
\emph{Case 3:}
$(n_0+1)\leq m\leq (n_1+1)(n_0+1)-1$.
\par\smallskip
For these values of $m$  we have $Bz^m=\varepsilon z^{m+n_0+1}$ so $N_j(Bz^m)=\varepsilon j^{m+n_0+1}$.

If $j\geq j_0$, we have  $j^{n_0+1}\leq M_j$, $k_j>j$ and $\varepsilon<1$, so $N_j(Bz^m)\leq M_j k_j^m$.

If $1\le j<j_0$, 
\begin{align*}
N_j(Bz^m)&\leq \varepsilon j_0^{n_0+1} j^m
\leq M_j k_j^m
\end{align*}
where the last inequality follows from \ref{ii} and from the fact that $j<k_j$. 
\par\medskip
\emph{Case 4:}
$n_2(n_0+1)\leq m\leq (n_2+1)(n_0+1)-1$.
\par\smallskip
We write $m$ as $m=n_2(n_0+1)+l$ with $0\leq l\leq n_0$. In this case we have that
\begin{align*}
Bz^m&=-\sum_{i=0}^{n_1} \frac1{\varepsilon^{n_1-i+1}} z^{(n_0+1)i} A_0 ^{n_2-i+1} z^l -\sum_{i=n_1+1}^{n_2} z^{(n_0+1)i} A_0^{n_2-i+1} z^l
\end{align*}
and since $0<\varepsilon<1$, it follows that
\begin{align*}
N_j(Bz^m)&\leq \frac1{\varepsilon^{n_1+1}} \sum_{i=0}^{n_2 } N_j\left( z^{(n_0+1)i} A_0 ^{n_2-i+1} z^l\right)
= \frac1{\varepsilon^{n_1+1}} \sum_{i=0}^{n_2 } j^{(n_0+1)i} N_j\left(A_0 ^{n_2-i+1} z^l\right).
\end{align*}
\par\medskip
%%%%
%%%%
{Next we will estimate} the norms of the iterates of $A_0$. The main difficulty with this estimate comes from the fact that given a polynomial $P$, the quantity $N_j(A_0P)$ can be controlled by $N_{k_j}(P)$ only, and not by $N_j(P)$.
\par\smallskip
We remark that  if $P$ is a polynomial of degree at most $d$, we have $N_j(P)\leq j^d N_1(P)$. Thus for any polynomial $P$ of degree at most $n_0$, we have  $N_1(A_0P)\leq M_1k_1^{n_0} N_1(P)$. It follows that
\begin{align}
N_1(A_0^nP)&\leq \left(M_1 k_1^{n_0}\right)^n N_1(P)\quad\textrm{ for every } n\ge 0. \label{eq3}
\end{align}
\par\smallskip
-- We consider first the case where $j\geq j_0$. For every $0\le i\le n_2$ and every $0\le l' \le n_0$, (\ref{eq3}) gives
\begin{align*}
 N_j\left(A_0^{n_2-i+1}z^{l'}\right)
 &\leq j^{n_0} N_1\left(A_0^{n_2-i+1}z^{l'}\right)\leq j^{n_0}\left( M_1k_1^{n_0}\right)^{n_2-i+1}.
\end{align*}
It then follows that 
\begin{align*}
N_j\left(Bz^m\right) 
&\leq \frac{1}{\varepsilon^{n_1+1}}\sum_{i=0}^{n_2} j^{(n_0+1)i} j^{n_0} (M_1k_1^{n_0})^{n_2-i+1}\\
&= \frac{j^{(n_0+1)(n_2+1)+n_0}}{\varepsilon^{n_1+1}} \sum_{i=0}^{n_2} \left( M_1k_1^{n_0}\frac1{j^{n_0+1}}\right)^{n_2-i+1}.
\end{align*}
By the definition  of $c_0$ given in \ref{iii},  we get
\begin{align*}
N_j\left(Bz^m\right) &\leq \frac{j^{(n_0+1)(n_2+1)+n_0}}{\varepsilon^{n_1+1}} c_0 M_1k_1^{n_0}\frac1{j^{n_0+1}}\cdot
\end{align*}
Since $k_j\geq 1$ and $m\geq n_2(n_0+1)$ it follows that 
\begin{align*}
N_j\left(Bz^m\right)&\leq M_jk_j^m c_0 \frac{1}{\varepsilon^{n_1+1}} M_1k_1^{n_0}
\left(\frac{j}{k_j}\right)^{n_2(n_0+1)} \frac{j^{n_0}}{M_j},
\end{align*}
and using the fact that $M_j\geq j^{n_0+1}$ (because $j\geq j_0$), it follows that
\begin{align*}
N_j\left(Bz^m\right)&\leq M_jk_j^m c_0 \frac{1}{\varepsilon^{n_1+1}} M_1k_1^{n_0}
\left(\frac{j}{k_j}\right)^{n_2(n_0+1)} \frac{1}{j}\cdot
\end{align*}

\par\smallskip
If $j\geq j_1$, using that $\frac{j}{k_j}\leq 1$ as well as property \ref{iii}, we get that
$
N_j\left(Bz^m\right) \leq M_j k_j^m.
$

\par\smallskip
If $j_0\leq j<j_1$, \ref{iv}-\ref{iva} gives the inequality $N_j\left(Bz^m\right) \leq M_j k_j^m.$

\par\smallskip
-- It now remains to consider the integers $j$ such that $1\leq j<j_0$. For all polynomials $P$ of degree at most $n_0$ we have 
\begin{align*}
N_j(A_0P)&\leq (1-\delta) M_j N_{k_j}(P)
\leq (1-\delta) M_j {k_j}^{n_0} N_1(P),
\end{align*}
from which it follows that
\begin{align*}
 N_j(Bz^m)
 &\leq \frac1{\varepsilon^{n_1+1}}\sum_{i=0}^{n_2} j^{(n_0+1)i}N_j\left(A_0^{n_2-i+1}z^l\right)\\
 &\leq \frac1{\varepsilon^{n_1+1}}\sum_{i=0}^{n_2} 
 j^{(n_0+1)i} (1-\delta)  M_j {k_j}^{n_0} N_1\left(A_0^{n_2-i}z^l\right).
\end{align*}
We use   $(\ref{eq3})$ in order to get $N_1\left(A_0^{n_2-i}z^l\right)\leq (M_1 k_1^{n_0})^{n_2-i}$ and 
 \begin{align*}
 N_j(Bz^m) 
 &\leq \frac1{\varepsilon^{n_1+1}} M_j {k_j}^{n_0}\sum_{i=0}^{n_2} 
 j^{(n_0+1)i} \left(M_1k_1^{n_0}\right)^{n_2-i}\\
 &=\frac1{\varepsilon^{n_1+1}} M_j {k_j}^{n_0} j^{n_2(n_0+1)}\sum_{i=0}^{n_2} 
  \left(M_1k_1^{n_0}\frac1{j^{n_0+1}}\right)^{n_2-i}.
 \end{align*}
If $M_1k_1^{n_0}\frac1{j^{n_0+1}}\leq 1$, we obtain
 \begin{align*}
 N_j(Bz^m) 
 &\leq\frac1{\varepsilon^{n_1+1}} M_j {k_j}^{n_0} j^{n_2(n_0+1)}(n_2+1)
 = M_j {k_j}^{m} \frac{n_2+1}{\varepsilon^{n_1+1}}  \frac{j^{n_2(n_0+1)}}{k_j^{n_2(n_0+1)+l-n_0}}\\
 &\leq M_j {k_j}^{m} \frac{n_2+1}{\varepsilon^{n_1+1}}  \left(\frac{j}{k_j}\right)^{n_2(n_0+1)} \left(\max_{1\le i\le  j_0} k_i\right)^{n_0}.
  \end{align*}
Then condition \ref{iv}-\ref{ivb}  implies that $N_j(Bz^m) \leq M_j k_j^{m}.$

\par\smallskip
In the case where $M_1k_1^{n_0}\frac1{j^{n_0+1}}\geq 1$, we have
 \begin{align*}
 N_j(Bz^m) 
 &\leq\frac1{\varepsilon^{n_1+1}} M_j {k_j}^{n_0} j^{n_2(n_0+1)}(n_2+1)\left(M_1k_1^{n_0} \frac1{j^{n_0+1}}\right)^{n_2}\\
 &= M_j {k_j}^{m} \frac{n_2+1}{\varepsilon^{n_1+1}}k_1^{n_0-l} \left(\frac{M_1}{k_1}\right)^{n_2}  \frac{k_1^{n_2(n_0+1)+l-n_0}}{k_j^{n_2(n_0+1)+l-n_0}}\\
 &\leq M_j {k_j}^{m} \frac{n_2+1}{\varepsilon^{n_1+1}}  \left(\frac{M_1}{k_1}\right)^{n_2} k_1^{n_0-l}
  \end{align*}
  where the last inequality comes from the fact that $k_j\geq k_1$.
Now \ref{iv}-\ref{ivc} gives that $N_j(Bz^m) \leq M_j k_j^{m}$. This shows  for every $j\ge 1$ and every $m\ge 0$ that $N_j(Bz^m) \leq M_j k_j^{m}$, and Fact \ref{fait4} is thus proved.
\end{proof}
Combining Facts \ref{fait3} and \ref{fait4} concludes the proof of Proposition \ref{approximationalgebrehypercyclique}.
\end{proof}

\textcolor{black}{
What we  actually need in the proof of Theorem \ref{maintheorem} is that operators of the form $B+\delta S_{n+1}$ are dense in $\clmk$, for $\delta\in(0,1)$ and $B\in\clmk$ nilpotent such that $B=T_nBT_n$. The following corollary shows that this is an easy consequence of Proposition \ref{approximationalgebrehypercyclique}. 
\begin{corollary}\label{corollaire}
 Let $(M_j)_{j\ge 1}$ be a sequence of positive real numbers, and let $(k_j)_{j\ge 1}$ be a sequence of positive integers such that 
 \begin{enumerate}[\normalfont (i)]
  \item for all $\alpha\ge 1$, $j^\alpha=o({M_j})$ as $j$ tends to infinity;
  \item $k_j>j$ for every $j\ge 1$;
  \item $M_1<k_1$ and $k_j\geq k_1$ for every $j\ge 1$.
 \end{enumerate}
Let $A$ be an operator belonging to $\clmk$ and let $\cv$ be an \emph{\sot}-neighborhood of $A$ in $\clmk$. Then there exist $B\in \clmk$, $\delta\in(0,1)$ and $n\ge 1$ such that $B=T_n B T_n$, $B^{n+1} =0$ and $B+\delta S_{n+1}$ belongs to $\cv$.
\end{corollary}
\begin{proof}
    We denote by $B'\in\cv$ a nilpotent operator with $B'=T_nB'T_n$ for some integer $n\geq 1$, given by Proposition \ref{approximationalgebrehypercyclique}. If $\delta\in(0,1)$ is sufficiently small, $(1-\delta)B' + \delta S_{n+1}$ belongs to $\cv$, $B=(1-\delta) B'$ belongs to $\clmk$ and it satisfies $B=T_n BT_n$ and $B^{n+1}=0$.
\end{proof}}

%%%%%%%%%%%%%%%%%%%%%%%%%%%%%%%%%%%%%%%%%%%%%%%%%%%%%%%%%%%%%%%%%%%%

\section{Density of hypercyclic operators}\label{section3}

%%%%%%%%%%%%%%%%%%%%%%%%%%%%%%%%%%%%%%%%%%%%%%%%%%%%%%%%%%%%%%%%%%%%

Our aim in this section is to prove the following theorem.
\begin{theorem}\label{maintheorembis}
Let $(M_j)_{j\ge 1}$ be a sequence of positive real numbers, and let $(k_j)_{j\ge 1}$ be a sequence of positive integers such that  $M_j\geq j+1$ and $k_j\geq j+2$ for every $j\ge 1$. Then the set of hypercyclic operators  is a dense $G_\delta$ subset of $(\clmk,\emph{\sot})$.
\end{theorem}

The statement of Theorem \ref{maintheorembis} is not strictly necessary for the proof of Theorem \ref{maintheorem}. However, we will subsequently make use of many of the ingredients of its proof, in particular of the expression and the properties of the eigenvectors of operators
of the form $B+\delta S_{n+1}$, where $B=T_n B T_n$  (cf.\ Proposition \ref{approximationhypercyclique}).
\begin{proof}
Let us begin by proving that the set of hypercyclic operators is $G_\delta$ in $\clmk$. The argument for this is classical (cf.\ \cite{GMM1}*{Corollary 2.2}). 
{Let $(\cu_p)_{p\geq 1}$ be a countable basis of open sets of $\oc$.} The set of hypercyclic operators in $\clmk$ can be expressed as
{\begin{align*}
\cg=& \left\{\vphantom{T^N(u^n) }
T\in \clmk\ ; \ \forall p,q\geq 1\; \exists u\in\cu_p\;\exists N\geq 1 \text{ such that }  T^N u\in\cu_q
\right\}\\
=& \bigcap_{p,q\geq 1} \bigcup_{N\geq 1} \left\{\vphantom{T^N(u^n) }
T\in \clmk\ ; \ \exists u\in\cu_p \text{ such that }  T^N u\in\cu_q \right\}
.
\end{align*}
To establish that $\cg$ is a $G_\delta$ set, it suffices to prove that, for all $p,q$ and $N$, the set $$\cg_{p,q,N}\coloneqq\left\{\vphantom{T^N(u^n) }
T\in \clmk\ ; \ \exists u\in\cu_p \text{ such that }  T^N u\in\cu_q \right\}$$ is \sot-open. This is a consequence of the following claim:
% ??(cf.\ \cite{GMM1}*{Lemma 2.1} for the Banach space analogue)??.
\begin{claim}\label{claim}
For every integer $N\geq 1$ and every function $u\in\oc$, the map $\phi_{N,u}:T\mapsto T^Nu$ from $\clmk$ into $\oc$ is \emph{\sot}-continuous.
\end{claim}
We take the claim for granted for the moment. Let $T_0$ belong to $\cg_{p,q,N}$. Choose $u\in\cu_p$ such that $T_0^N u$ belongs to $\cu_q$. Then $\phi^{-1}_{N,u}(\cu_q)$ is an \sot-open set contained in $\cg_{p,q,N}$ containing $T_0$. Therefore $\cg_{p,q,N}$ is indeed \sot-open, and $\cg$ is an \sot-$G_\delta$ set, provided the claim holds.

\begin{proof}[Proof of Claim \ref{claim}]
Let $(T_n)$ be a sequence of elements of $\clmk$ that \sot-converges to $T$.
For every $j\ge 0$, define $\kappa(j) \coloneqq k_j$. Note that for all $S\in\clmk$,  all $l\geq 1$, all $f\in\oc$ and all $j\geq 1$, we have
\begin{align}
    N_j(S^lf)&\leq \left(\prod_{m=1}^{l} M_{\kappa^{m-1}(j)}\right) N_{\kappa^l(j)}(f), \label{eqnouvelle}
\end{align}
where $\kappa^m(j)=\kappa\circ\cdots \circ \kappa \,(j)$, $j\ge 0$, is the composition of the function $\kappa$ with itself $m$ times.
Now we express $\phi_{N,u}(T_n)-\phi_{N,u}(T)$ as follows:
\begin{align*}
    \phi_{N,u}(T_n)-\phi_{N,u}(T)
    = \sum_{l=1}^N \left(T_n^lT^{N-l} u - T_n^{l-1} T^{N-l+1} u\right)
    =\sum_{l=1}^N T_n^{l-1} (T_n-T) T^{N-l} u.
\end{align*}
Thus, by (\ref{eqnouvelle}), we have for all $j\ge 1$ that
\begin{align*}
    N_j\left(\phi_{N,u}(T_n)-\phi_{N,u}(T)\right)
    \le &\sum_{l=1}^N N_j\left(T_n^{l-1} (T_n-T) T^{N-l} u\right)\\
    \le & \sum_{l=1}^N \left(\prod_{m=1}^{l-1} M_{\kappa^{m-1}(j)}\right) N_{\kappa^{l-1}(j)}\left((T_n-T)T^{N-l}u\right).
\end{align*}
Since $N\ge 1$ and $u\in\oc$ are fixed, and since the sequence $(T_n)_n$ \sot-converges to $T$, we have $ N_{\kappa^{l-1}(j)}\left((T_n-T)T^{N-l}u\right)\longrightarrow 0$ as $n\to +\infty$ for every $l=1,\ldots, N$. Thus $\ N_j\left(\phi_{N,u}(T_n)-\phi_{N,u}(T)\right) \longrightarrow 0$, proving the claim.
\end{proof}
}

\par\smallskip
Next we prove that hypercyclic operators are \sot-dense in $\clmk$. To this end, by Proposition \ref{approximationhypercyclique} it suffices to
prove that the operators $B+\delta S_{n+1}$, with $n\ge 0$ and $B=T_n B T_n$, are hypercyclic. 

The general strategy of the proof is the following: we will show (cf.\ Lemma \ref{lemme2} below) that if a subset $\ca$ of $\cc$ has an accumulation point in $\cc$, then the linear span of entire functions $f$ which satisfy $(B+\delta S_{n+1})f=\lambda f$ for some $\lambda \in\ca$ is dense in $\oc$. The Godefroy-Shapiro Criterion (cf.\ for instance \cite{GroPer}*{Theorem 3.1}) will then imply that $B+\delta S_{n+1}$ is hypercyclic, and thus Theorem \ref{maintheorembis} will follow.
\par\smallskip
Our first task is to describe the eigenvectors of $B+\delta S_{n+1}$. For $n\ge 0$, we denote by $\Pi_{n+1} :\cc^{n+1}\to \cc$ the canonical projection defined by $\Pi_{n+1}(z_0,\ldots, z_n)=z_n$ for every $(z_0,\ldots, z_n)\in\cc^{n+1}$. For $\lambda\in\cc$, let $M_\lambda$ be the $(n+1)$-square matrix defined as
\begin{align*}
 M_\lambda&=
 \begin{pmatrix}
  0     &\cdots &\cdots &0     &\frac{(n+1)!}{\delta} \lambda\\
  1     &\ddots &       &0     &0\\
  0     &\ddots &\ddots &      &\vdots\\
  \vdots&\ddots & \ddots&\ddots&\vdots\\
  0     &\cdots &0      &1      &0
   \end{pmatrix}.
 \end{align*}

We begin by proving the following lemma.
 
\begin{lemma}\label{lemme1}
Let $B$ be a continuous linear operator on $\oc$ and let $\delta>0$. Suppose that there exists $n\ge 0$ such that $B=T_n B T_n$.
Then for every polynomial $P\in\cc_n[z]$ and every $\lambda\in\cc$, the entire function $f_{\lambda, P}$ defined by 
\begin{align*}
 f_{\lambda,P}(z)=
 &P(z)+\Pi_{n+1}\left(
 \int_{[0,z]} \frac{(n+1)!}{\delta} \exp((z-t)M_\lambda) 
 \begin{pmatrix}
    (-B+\lambda I)P(t),0,\ldots,0
    \end{pmatrix}
  %\begin{pmatrix}(-B+\lambda I)P(t)\\ 0\\ \vdots\\  0\end{pmatrix}
  \mathrm{d}t
 \right)
\end{align*}
is the only eigenvector of $B+\delta S_{n+1}$ associated to the eigenvalue $\lambda$ such that $T_n f_{\lambda, P}=P$. In particular, $\dim\ker(B+\delta S_{n+1}-\lambda I)=n+1$.
\end{lemma}

\begin{proof}[Proof of Lemma \ref{lemme1}] Let $f$ belong to $\oc$. We set $P=T_nf$ and $g=f-P$. Since $B=T_n B T_n$ and  $T_ng=0$, we have that $Bg=0$. Therefore $(B+\delta S_{n+1})f=\lambda f$ if and only if $BP+\delta S_{n+1} g=\lambda g+\lambda P$. 

Thus $g$ is a solution of the ordinary differential equation 
$$(E) :\begin{cases}
   y^{n+1}-\lambda \frac{(n+1)!}\delta y=\frac{(n+1)!}\delta ( \lambda I -B)P\\
   y^{(n)}(0)=\cdots=y(0)=0.
  \end{cases}$$
\par\smallskip
  Hence there exists a unique eigenvector $f_{\lambda,P}$ of $B+\delta S_{n+1}$ associated to the eigenvalue $\lambda$ such that $T_n f_{\lambda, P}=P$. It has the form $f_{\lambda,P}=P+g$, where $g$ is the unique solution of $(E)$.
\par\smallskip
As usual, in order to solve the differential equation $(E)$, one sets $Y=(y^{(n)},\ldots, y)$. Then $y$ is a solution of $(E)$ if and only if $Y$ is a solution of the system
$$(E') :\begin{cases}
   Y'= M_\lambda Y+ \frac{(n+1)!}{\delta} 
   \begin{pmatrix}(-B+\lambda I)P(t),0,\ldots,0\end{pmatrix}\\
    %\begin{pmatrix}(-B+\lambda I)P(t)\\ 0\\ \vdots\\  0\end{pmatrix}\\
   Y(0)=0.
   \end{cases}$$
The unique solution of $(E')$ is given by
$$Y(z)=
 \int_{[0,z]} \frac{(n+1)!}{\delta} \exp((z-t)M_\lambda) 
 \begin{pmatrix}
    (-B+\lambda I)P(t),0,\ldots,0
    \end{pmatrix}
  %\begin{pmatrix}(-B+\lambda I)P(t)\\ 0\\ \vdots\\  0\end{pmatrix}
  \mathrm{d}t$$
which yields the expression of $f_{\lambda,P}$ as given in the statement of Lemma \ref{lemme1}.
\end{proof}

We will need a more explicit expression of $f_{\lambda,P}$ when $P$ describes the canonical basis $P_0,\ldots, P_n$ of $\cc_n[z]$, with $P_l(z)=z^l$, $l=0,\ldots, n$. For $i,j=0,\ldots, n$, let $a_{ij}\in\cc$ be such that
$$Bz^j=\sum_{i=0}^n a_{ij}z^i\quad \textrm{ for every } 0\le j\le n.$$ 
We also set, for $j=0,\ldots, n$ and $l\ge 1$,  $$\alpha^{(j)}_l=\sum_{m=0}^j \binom{j}{m} \frac{(-1)^m}{(n+1)l+m}\quad\textrm{ and }\quad\alpha^{(j)}_0=1.$$

\begin{lemma}\label{calculVP}
 Let $B$ be a continuous linear operator on $\oc$ and let $\delta>0$. Suppose that there exists $n\ge 0$ such that $B=T_n B T_n$.
  Then we have for every $i=0,\ldots, n$
  %%%
  %%%
  {
 \begin{align*}
  f_{\lambda,P_i}(z)=z^i+
  \sum_{l=0}^{+\infty} \left( \left(\frac{(n+1)!}{\delta} \right)^{l+1} \right. 
  \frac{1}{\left((n+1)l+n\right)!} 
  \left( \vphantom{\sum_{j=0}^n}\right. &\lambda^{l+1}\alpha^{(i)}_{l+1}z^{i+(n+1)(l+1)} \\
  &\;\left.\left.-\lambda^l \sum_{j=0}^n a_{ji} z^{(n+1)(l+1)+j}\alpha_{l+1}^{(j)}\right)
  \right).
 \end{align*}
}
\end{lemma}

\begin{proof}[Proof of Lemma \ref{calculVP}] We define the matrices $J^{(j)}=\left(J_{m,l}^{(j)}\right)_{1\leq m,l\leq n+1}$, for $j=0,\ldots, n,$ by setting
\begin{align*}
J_{m,l}^{(j)}=
\begin{cases}
 1, &\text{for } m=j+1,\ldots, n+1 \text{ and } l=m-j\\
 0, &\text{otherwise}
\end{cases}
\end{align*}
and  $N^{(j)}=\left(N_{m,l}^{(j)}\right)_{1\leq m,l\leq n+1}$, for $j=1,\ldots, n,$ by setting
\begin{align*} 
N_{m,l}^{(j)}=
\begin{cases}
 1, &\text{for } m=1,\ldots, n+1-j \text{ and } l=j+m\\
 0, &\text{otherwise}
\end{cases}
\end{align*}
with  $N^{(n+1)}=0$.

We then have, for every $j=0,\ldots, n$ and every $l\ge 0$ that
$$M_\lambda^{(n+1)l+j}=\left(\lambda\frac{(n+1)!}{\delta}\right)^l \left(J^{(j)} + \lambda\frac{(n+1)!}{\delta} N^{(n+1-j)}\right),$$
which implies that
\begin{align*}
 &\exp\left((z-t)M_\lambda\right)=\sum_{l=0}^{+\infty}\left( \sum_{j=0}^n\frac{(z-t)^{(n+1)l+j}}{((n+1)l+j)!}
 \left(\lambda\frac{(n+1)!}{\delta}\right)^l \left(J^{(j)} + \lambda\frac{(n+1)!}{\delta} N^{(n+1-j)}\right)
 \right).
\end{align*}
It then follows from Lemma \ref{lemme1} that
\begin{align*}
 f_{\lambda,P_i}(z)&=z^i+\int_{[0,z]} \sum_{l=0}^{+\infty} 
 \frac{(z-t)^{(n+1)l+n}}{((n+1)l+n)!}\left(\frac{(n+1)!}{\delta}\right)^{l+1}\lambda^l\left(\lambda t^i-\sum_{j=0}^n a_{ji}t^j\right)\mathrm{d}t.
\end{align*} 
In order to compute the above integral, we decompose $t^j$ as $t^j=\sum_{m=0}^j
\binom{j}{m}z^{j-m}(t-z)^{m}$ and thus obtain that {the integral in the expression above is equal to}
%%%
%%%
{
\begin{align*}
\int_{[0,z]} \sum_{l=0}^{+\infty} \left(
 \frac{(-1)^{(n+1)l+n}}{((n+1)l+n)!}\left(\frac{(n+1)!}{\delta}\right)^{l+1} \right. &
  \lambda^l 
  \left(\lambda \sum_{m=0}^i\binom {i}{m} z^{i-m} (t-z)^{(n+1)l+n+m}
  \vphantom{-\sum_{j=0}^n a_{ji}\sum_{m=0}^j \binom jm {z^{j-m}(t-z)^{(n+1)l+n+m}}}\right.\\
  &\left.\left.
  \;\; -\sum_{j=0}^n a_{ji}\sum_{m=0}^j \binom jm {z^{j-m}(t-z)^{(n+1)l+n+m}}\right)\right)\mathrm{d}t\\
  =\sum_{l=0}^{+\infty} \left(
 \frac{(-1)^{(n+1)(l+1)}}{((n+1)l+n)!}\left(\frac{(n+1)!}{\delta}\right)^{l+1} \right. &
  \lambda^l 
  \left(\lambda \sum_{m=0}^i\binom {i}{m} \frac{z^{i-m} (-z)^{(n+1)(l+1)+m}}{(n+1)(l+1)+m}
  \vphantom{-\sum_{j=0}^n a_{ji}\sum_{m=0}^j \binom jm {z^{j-m}(t-z)^{(n+1)l+n+m}}}\right.\\
  &\left.\left.
  \quad -\sum_{j=0}^n a_{ji}\sum_{m=0}^j \binom jm \frac{{z^{j-m}(-z)^{(n+1)(l+1)+m}}}{(n+1)(l+1)+m}\right)\right).
  \end{align*}
}
%%%
%%%
  This then yields that
  \begin{align*}
 f_{\lambda,P_i}(z)=
z^i+\sum_{l=0}^{+\infty} \left(
 \frac{1}{((n+1)l+n)!}\left(\frac{(n+1)!}{\delta}\right)^{l+1} \right. &
  \left(\alpha_{l+1}^{(i)}\lambda^{l+1}z^{i+(n+1)(l+1)} 
  \vphantom{-\lambda^l\sum_{j=0}^n \alpha_{l+1}^{(j)}a_{ji}z^{(n+1)(l+1)+j}}\right.\\
  &\left.\left.
  \quad -\lambda^l\sum_{j=0}^n \alpha_{l+1}^{(j)}a_{ji}z^{(n+1)(l+1)+j}\right)\right),
\end{align*} 
which is exactly the expression given in Lemma \ref{calculVP}.
\end{proof}

Our last task is to prove the density of the vector space spanned by the functions $f_{\lambda,P}$, where $P$ ranges over $P\in\cc_{n}[z]$ and $\lambda$ ranges over a  certain subset $\ca$ of $\cc$ having an accumulation point in $\cc$.
\begin{lemma}\label{lemme2}
 Let $B$ be a continuous linear operator on $\oc$ such that $B=T_n B T_n$ for some $n\ge 0$ and let $\delta>0$. For every subset $\ca$ of $\cc$ which has an accumulation point in $\cc$, the vector space $$\vspan{\bigcup_{\lambda\in \ca} \ker(B+\delta S_{n+1} -\lambda I)}$$ is dense in $\oc$.
\end{lemma}

\begin{proof}[Proof of Lemma \ref{lemme2}] We proceed as in the proof of \cite{GodSha}*{Theorem 5.1}. Let $\Lambda$ be a continuous linear functional on $\oc$ that vanishes on each function $f_{\lambda,P}$, $\lambda\in\ca$, $P\in\cc_n[z]$. The lemma will be proved as soon as we show that $\Lambda=0$. As in \cite{GodSha}, we use the fact that there exists a complex Borel measure $\mu$ on $\C$, supported on a disk of radius $r>0$ and centered at the origin, such that for every $f\in\oc$
$$\Lambda(f)=\int_{\C} f\mathrm{d}\mu.$$
\par\smallskip
For $0\le j\le n$ and $l\ge 0$, we set
$$u_l^{(j)}=\int_{\C} z^{(n+1)l+j} \mathrm{d}\mu(z).$$
\par\smallskip
We aim to prove for all $j$ and $l$, that $u_l^{(j)}=0$. The continuity of $\Lambda$ will then imply that $\Lambda=0$. 

\par\smallskip
For every polynomial $P\in\cc_n[z]$, we define $F_P\colon \cc\to\cc$ by setting
$$F_P(\lambda)=\int_{\C} f_{\lambda,P}\mathrm{d}\mu, \quad\lambda\in\C.$$
On the one hand, since $\Lambda(f_{\lambda,P})=0=\int_{\C} f_{\lambda,P}\mathrm{d}\mu$ for every $\lambda\in\ca$ and every $P\in\cc_n[z]$, we have $F_P(\lambda)=0$ for every $\lambda\in\ca$.
\par\smallskip
On the other hand, since $f_{\lambda,P}$ depends holomorphically on $\lambda$ by Lemma \ref{lemme1}, and since $\mu$ is compactly supported, differentiation under the integral sign shows that $F_P$ is an entire function. Since $\ca$ has an accumulation point, we deduce that $F_P=0$. In particular, for every $P\in \cc_n[z]$ and every $l\ge 0$, $$F^{(l)}_P(0)=\int_{\C} \left.
\frac{\partial^l f_{\lambda,P}}{\partial \lambda^l}  
\right|_{\lambda=0}\mathrm{d}\mu=0.$$ 
\par\smallskip
We will now show that the values of the $u_l^{(j)}$'s are linked to those of the derivatives  $F^{(l)}_{z^j}(0)$. In order to do this, we need to compute the successive derivatives of $f_{\lambda,z^j}$ with respect to $\lambda$ at the point $\lambda=0$.
\par\smallskip
We continue to denote by $B$ the $(n+1)$-square matrix $B=(a_{ij})_{1\leq i,j\leq n+1}$, i.e.\ the matrix of the restriction of the operator $B$ to $\cc_{n}[z]$ with respect to the canonical basis of $\cc_n[z]$. We also denote by $\tilde U_l$ the vector $\cc^{n+1}$ defined by $\tilde U_l=(\alpha_l^{(0)} u_l^{(0)}, \alpha_l^{(1)} u_l^{(1)},\ldots, \alpha_l^{(n)} u_l^{(n)})$. From Lemma \ref{calculVP} we deduce that
\begin{align*}
 F_{P_i}(0)&=u^{(i)}_0-\frac{(n+1)!}\delta \frac1{n!} \sum_{j=0}^n a_{ji}\alpha_1^{(j)}u_1^{(j)}\quad\textrm{ for every } i=0,\ldots n,
\end{align*}
and since $F_{P_i}(0)=0$ for $i=0,\ldots, n$ we get that
\begin{align*}
 \tilde U_0 = \frac{(n+1)!}{\delta n!} {^\intercal}B\tilde U_1,
\end{align*}
where the notation ${^\intercal}B$ denotes the transpose of the matrix $B$.
For $l\geq 1$, we differentiate with respect to $\lambda$ the expression given in Lemma \ref{calculVP} and obtain that for every $i=0,\ldots, n$,
\begin{align*}
 \diffp{^l F_{P_i}}{\lambda^l}(0)&
 =\left(\frac{(n+1)!}{\delta}\right)^l \left( \frac{l!}{((n+1)(l-1)+n)!} 
 \alpha_l^{(i)}u_l^{(i)}\right.\\
 &\quad\qquad\left.
 -\frac{(n+1)!}{\delta} \frac{l!}{((n+1)l+n)!} \sum_{j=0}^n a_{ji} \alpha_{l+1}^{(j)}u_{l+1}^{(j)}\right) =0
 \end{align*}
from which we deduce that
\begin{align*}
 \tilde U_l = \frac{(n+1)!}{\delta }\frac{\left((n+1)(l-1)+n\right)!}{\left((n+1)l+n\right)!} {^\intercal}B\tilde U_{l+1}.
\end{align*}
It follows that for every $m\ge 0$ and every $l\ge 0$,
\begin{align*}
 \tilde U_l &= \left(\frac{(n+1)!}{\delta }\right)^{m}\frac{\left((n+1)(l-1)+n\right)!}{\left((n+1)(l+m-1)+n\right)!} {^\intercal}B^m\tilde U_{l+m},
\end{align*}
with the convention that $(-1)!=1$.
\par\smallskip
If we denote by $\norm{\,\cdot\,}$ the sup norm on $\cc^{n+1}$ and by $\vertiii\cdot$ the norm on $(n+1)$ square matrices subordinated to this norm, we get that
\begin{align*}
 \norm{\tilde U_l}& \leq
 \left(\frac{(n+1)!}{\delta }\right)^{m}\frac{\left((n+1)(l-1)+n\right)!}{\left((n+1)(l+m-1)+n\right)!} \vertiii{{^\intercal}B}^m \norm{\tilde U_{l+m}}\quad\textrm{ for every }l,m\ge 0.
\end{align*}
\par\smallskip
We now observe that for every $j=0,\ldots, n$ and every $l\ge 1$, we have
\begin{align*}
 \alpha_l^{(j)}&=\sum_{m=0}^j \binom jm \frac{(-1)^m}{(n+1)l+m}
 =\int_0^1 \sum_{m=0}^j \binom jm (-1)^m t^{(n+1)l+m-1} \mathrm{d}t\\
 &=\int_0^1 t^{l(n+1)-1}(t-1)^j \mathrm{d} t,
 \end{align*}
so that $0\leq \alpha_l^{(j)} \leq 1$ for every $l\ge 1$. Recall also that $\alpha_0^{(j)}=1$.
\par\smallskip
On the other hand, remembering that $\mu$ is supported on the disc centered at $0$ and of radius $r$ (which may be assumed to be bigger than 1), we have
\begin{align*}
 \abs{u_l^{(j)}}=\abs{\int_{\C} z^{(n+1)l+j}\mathrm{d}\mu(z)}
 \leq r^{(n+1)l+j} |\mu|(\cc) \quad\textrm{ for every } j=0,\ldots, n \textrm{ and } l\ge 0.
\end{align*}
It thus follows that
\begin{align*}
 \norm{\tilde U_l}& \leq
 \left(\frac{(n+1)!}{\delta }\right)^{m}\frac{\left((n+1)(l-1)+n\right)!}{\left((n+1)(l+m-1)+n\right)!} \vertiii{{^\intercal}B}^m r^{(n+1)(l+m)+n} |\mu|(\cc).
\end{align*}
Letting $m$ go to infinity, we conclude that $\tilde U_l=0$. Thus $u^{(j)}_l=0$ for all $j,l\ge 0$,  from which it follows that $\Lambda=0$.
Lemma \ref{lemme2} is thus proved.
\end{proof}

As explained at the beginning of the proof of Theorem \ref{maintheorembis}, Lemma \ref{lemme2} combined with the Godefroy-Shapiro Criterion allows us to conclude that operators in $\clmk$ of the form $B+\delta S_{n+1}$, with $\delta>0$, $n\ge 0$ and $B=T_n B T_n $, are hypercyclic. Hypercyclic operators are thus dense in $(\clmk,{\sot})$. We have already observed that the set of hypercyclic operators is $G_\delta$ in  $(\clmk,{\sot})$, so this concludes the proof of Theorem \ref{maintheorembis}.
\end{proof}

%%%%%%%%%%%%%%%%%%%%%%%%%%%%%%%%%%%%%%%%%%%%%%%%%%%%%%%%%%%%%%%%%%%%

\section{Operators supporting a hypercyclic algebra: proof of Theorem \ref{maintheorem}}\label{section4}

%%%%%%%%%%%%%%%%%%%%%%%%%%%%%%%%%%%%%%%%%%%%%%%%%%%%%%%%%%%%%%%%%%%%

A last crucial step in the proof of Theorem \ref{maintheorem} is the following result.

\begin{theorem}\label{operateur_algebre_hypercyclique}
Let $B$ be a continuous linear operator on $\oc$ such that there exists $n\ge 0$ which satisfies $B=T_n B T_n$. Suppose also for every $j\ge 0$ that the sequence $(B^m z^j)_{m}$ converges to $0$ in $\oc$ as $m$ goes to infinity. Then for every $\delta\in (0,1)$, the operator $A=B+\delta S_{n+1}$ supports a hypercyclic algebra.
\end{theorem}

To prove Theorem \ref{operateur_algebre_hypercyclique}, we will apply \cite{Bay1}*{Lemma 1.6} (cf.\ also  \cite{BCP}*{Corollary 2.4}), which provides a useful ``Birkhoff-type Criterion'' for proving that an operator admits a hypercyclic algebra.

\begin{lemma}[\cite{Bay1}]\label{lemme1.6}
 Let $A$ be a continuous operator on a separable, metrizable  and complete topological algebra $X$. Assume that for any pair $(\cu,\cv)$ of nonempty open sets in $X$, any neighborhood $\cw$ of $0$ in $X$, and for any integer $m\ge 1$, one can find a vector $u\in\cu$ and an integer $N\ge 0$ such that $A^N (u^j)$ belongs to $\cw$ for every $1\le j<m$ and $A^N(u^m)$ belongs to $\cv$. Then $A$ supports a hypercyclic algebra.
\end{lemma}

In \cite{Bay1}, this criterion is applied in the following way: the vector $u$ is found in a dense linear subspace of the space $X$ spanned by holomorphic vector fields $(E_\lambda)_{\lambda\in\cc}$ of eigenvectors of $A$. To compute $A^N(u^j)$, we require that these eigenvector fields be multiplicative, i.e.\ that  $E_{\lambda+\mu}=E_\lambda \cdot E_\mu$ for every $\lambda$ and $\mu$ in $\cc$.  In our case, the product of two eigenvectors is not necessarily an eigenvector. To overcome this obstacle, we will introduce multiplicative vector fields whose powers are close to powers of eigenvectors fields. It is here that the assumption that the sequences $(B^m z^j)_{m}$, $j\ge 0$, converge to $0$ in $\oc$ plays a crucial role. 
\par\smallskip

We begin by giving  a simpler expression for some of the eigenvectors exhibited in Lemma \ref{lemme1}. 
% \par\smallskip
Let $B$ be an operator satisfying the hypothesis of Theorem \ref{operateur_algebre_hypercyclique} and let $A$ be the operator $A=B+\delta S_{n+1}$, with $\delta\in(0,1)$. For every $\alpha\in\cc$, we denote by $e_\alpha$ the function defined  by $e_\alpha(z)=\exp(\alpha z)$,
$z\in\cc$. We also set $$x_0=\left(\frac{(n+1)!}{\delta}\right)^{\frac1{n+1}},$$ and we denote by $\phi$ the function defined  by $$\phi(z)= \left(\frac{z}{x_0}\right)^{n+1}=\frac{\delta}{(n+1)!}\,z^{n+1} \quad\textrm{for every } z\in\cc.$$
\par\smallskip
For every $\beta \in\cc\setminus \phi^{-1}\left(\spectre B\right)$, where $\spectre B$ denotes the spectrum of the operator $B$ acting on $\oc$, we put $$\varepsilon_\beta=e_\beta+\left(\phi(\beta)I-B\right)^{-1}Be_\beta.$$ We note that $\spectre{B}$ is finite. 
Keeping in mind that $B\varepsilon_\beta$ is a polynomial of degree at most $n$, we have
\begin{align*}
 A\varepsilon_\beta
 &=\frac{\delta}{(n+1)!} \beta^{n+1} e_\beta+Be_\beta+B\left(\phi(\beta)I-B\right)^{-1}Be_\beta\\
 &=\phi(\beta)e_\beta+\phi(\beta)\left(\phi(\beta)I-B\right)^{-1}Be_\beta+Be_\beta -\left(\phi(\beta)I-B\right)\left(\phi(\beta)I-B\right)^{-1}Be_\beta\\
 &=\phi(\beta)\varepsilon_\beta,
\end{align*}
i.e.\ $\varepsilon_\beta$ is an eigenvector of $A$ associated to the eigenvalue $\phi(\beta)$.
\par\smallskip
We now prove the following density lemma.

\begin{lemma}\label{lemme3} 
Let $\cb$ be a subset of $\cc$ which has an accumulation point and is such that $\beta e^{\frac{2i\pi j}{n+1}}$ belongs to $\cb$ for every $0\le j\le n$ and every $\beta\in\cb$.
Then the linear vector space
$$\vspan{\varepsilon_\beta\ ; \ \beta\in\cb\setminus\phi^{-1}\left(\spectre{B}\right)}$$ is dense in $\oc$.
\end{lemma}

\begin{proof} We will show that for every $ \beta\in\cb\setminus\phi^{-1}\left(\spectre{B}\right)$, we have
$$\vspan{\varepsilon_{\beta'}\ ; \ \beta'=\beta e^{\frac{2i\pi j}{n+1}},\ j=0,\ldots, n}=\ker(A-\phi(\beta) I).$$ Then, since for every $\beta\in\cb\setminus \phi^{-1}\left(\spectre B\right)$ and every $0\le j\le n$ the complex numbers $\beta e^{\frac{2i\pi j}{n+1}}$ still belong to $\cb\setminus \phi^{-1}\left(\spectre B\right)$, it will follow that
$$\vspan{\bigcup_{\lambda \in\phi(\cb)\setminus \spectre B} \ker(A-\lambda I)}
=\vspan{\varepsilon_\beta\ ; \ \beta\in\cb\setminus\phi^{-1}\left(\spectre{B}\right)}.$$
Since $\phi(\cb)\setminus\spectre{B}$ has an accumulation point in $\cc$, Lemma \ref{lemme2}  will imply that the linear vector space spanned by the family $\left\{\varepsilon_\beta\ ; \ \beta\in\cb\setminus\phi^{-1}\left(\spectre{B}\right)\right\}$ is dense in $\oc$.
\par\smallskip
So, let $\beta$ belong to $\cb\setminus\phi^{-1}\left(\spectre{B}\right)$, and set $\beta_j=\beta e^{\frac{2i\pi j}{n+1}}$, $j=0,\ldots, n$. We will check that the functions $\varepsilon_{\beta_0},\ldots, \varepsilon_{\beta_n}$ are linearly independent in $\oc$. Since 
for every $j=0,\ldots, n$ the function $\varepsilon_{\beta_j}$ is an eigenvector of $A$ associated to the eigenvalue $\phi(\beta)$, and since, by Lemma \ref{lemme1}, $\dim\ker(A-\phi(\beta)I)=n+1$,  Lemma \ref{lemme3} will thus be proven.
\par\smallskip
Let $\gamma_0,\ldots, \gamma_n$ be $n+1$ complex numbers such that $\sum_{j=0}^n\gamma_j\varepsilon_{\beta_j}=0$. Since for every $j$ the function $\varepsilon_{\beta_j}$ is the sum of $e_{\beta_j}$ and a polynomial of degree at most $n$, differentiating $n+1$ times the latter sum we get that $\sum_{j=0}^n \beta^{n+1}\gamma_j e_{\beta_j}=0$. Since $\phi(0)=0$ is an eigenvalue of $B$, $\beta^{n+1}\neq 0$, and we thus have 
$\sum_{j=0}^n \gamma_j e_{\beta_j}=0$. Now since $e_{\beta_0},\ldots, e_{\beta_n}$ are linearly independent functions in $\oc$, $\gamma_0=\cdots=\gamma_j=0$, and $\varepsilon_{\beta_0},\ldots, \varepsilon_{\beta_n}$ are linearly independent. Lemma \ref{lemme3} is thus proved.
\end{proof}

The vectors $\varepsilon_\beta$, $\beta\in\cb$,   thus span a dense linear subspace of $\oc$ provided that $\cb$ has an accumulation point and is invariant by the rotation of angle $\frac{2\pi}{n+1}$. However,
in general it does not hold that $\varepsilon_\alpha\cdot \varepsilon_\beta=\varepsilon_{\alpha+\beta}$. In order to overcome this difficulty, we will take advantage of the fact that, of course, $e_\alpha\cdot e_\beta=e_{\alpha+\beta}$, and of the link between $A^m e_\beta$ and $A^m \varepsilon_\beta$ given by the following Fact \ref{faitIV-1}.

\begin{fact}\label{faitIV-1}
 For every $\beta \in\cc\setminus \phi^{-1}\left(\spectre B\right)$ and every $m\ge 0$, we have
 \begin{align}
  A^me_\beta&=\phi(\beta)^m \varepsilon_\beta -\left(\phi(\beta)I-B\right)^{-1}B^{m+1}e_\beta.\label{***}
 \end{align}
\end{fact}

\begin{proof} By definition of $\varepsilon_\beta$, we have the equality $e_\beta=\varepsilon_\beta - \left(\phi(\beta)I-B\right)^{-1}Be_\beta$. If we assume for some integer $m\ge 0$  that (\ref{***}) is true, then, since $\left(\phi(\beta)I-B\right)^{-1}B^{m+1}e_\beta$ is a polynomial of degree at most $n$, we have
\begin{align*}
  A^{m+1}e_\beta&=\phi(\beta)^m A\varepsilon_\beta -B\left(\phi(\beta)I-B\right)^{-1}B^{m+1}e_\beta\\
  &=\phi(\beta)^{m+1} \varepsilon_\beta -\left(\phi(\beta)I-B\right)^{-1}B^{m+2}e_\beta.
\end{align*}
Therefore, by induction, (\ref{***}) is valid for every $m\ge 0$.
\end{proof}

When $(B^mP)_m$ converges to 0 for every polynomial $P$ of degree at most $n$, Fact \ref{faitIV-1} implies that $A^m e_\beta$ is close to $A^m\varepsilon_\beta=\phi(\beta)^m \varepsilon_\beta$. This will be a key element in showing that $A$ satisfies the assumptions of Lemma \ref{lemme1.6}, cf.\ Lemma \ref{mainlemma} below. In the proof of this lemma, we will need Lemma \ref{lemmephi}, which can be compared to the beginning of the proof of Theorem 2.1 in \cite{Bay1}. 

\begin{lemma}\label{lemmephi}
Let $m$ be a positive integer, and let $\eta$ and $\varepsilon$ be positive real numbers. To the parameters $\eta$ and $\varepsilon$ we associate two subsets  $\ca$  and $\cb$ of $\cc$ defined as follows:
$\ca=D(0,\eta)$ is the open disk of center 0 and of radius $\eta$ in $\C$, and
$\cb$ is the set $$\cb=\bigcup_{j=1}^n \left\{\zeta_j(1+it)\ ; \ t\in[-\eta,\eta]\right\}$$ where $\zeta_j=x_0(1+\varepsilon)e^{\frac{2i\pi j}{n+1}}$ for every $0\le j\le n$.
%\par\smallskip
 Then there exist $\eta>0$ and $\varepsilon>0$ sufficiently small
  such that for every $j=1,\ldots, m$, every $d=0,\ldots, j$ with $d\neq m$, every $\tilde\xi_1,\ldots, \tilde\xi_{j-d}\in\ca$ and every $\xi_1,\ldots,\xi_d\in\cb$, we have
  \begin{align}
   \left|\phi\left(\tilde\xi_1+\cdots+\tilde\xi_{j-d}+\frac{\xi_1+\cdots+\xi_d}m\right)\right|&<1, \label{(i)}
  \end{align}
and for every $l=1,\ldots, m$ and every $\xi_1,\ldots,\xi_l\in\cb$ which are not all equal, we have
  \begin{align}
   \left|\phi\left(\frac{\xi_1+\cdots+\xi_l}m\right)\right|&<\left|\phi(\xi_1)\right|^{\frac1m}\cdots \left|\phi(\xi_l)\right|^{\frac1m}. \label{(ii)}
  \end{align}
\end{lemma}

\begin{proof} We first prove (\ref{(ii)}). We consider first the case where $l<m$: in this case we have
$$\left|\frac{\xi_1+\cdots +\xi_l}m\right|
\leq \frac{m-1}{m}x_0 (1+\varepsilon)(1+\eta)<x_0$$ as soon as $\varepsilon$ and $\eta$ are sufficiently small. Thus $$\left|\phi\left(\frac{\xi_1+\cdots+\xi_l}m\right)\right|<\phi(x_0)=1.$$ On the other hand, for every $i=1,\ldots, l$, 
$\left|\phi(\xi_i)\right|=(1+\varepsilon)^{n+1}>1$ for every $\varepsilon>0$. Thus  $$\left|\phi\left(\frac{\xi_1+\cdots+\xi_l}m\right)\right|<\left|\phi(\xi_1)\right|^{\frac1m}\cdots \left|\phi(\xi_l)\right|^{\frac1m}.$$
\par\smallskip
In the case where $l=m$, we denote by $i_j$ the unique integer with $0\le i_j \le n$ such that $\xi_j=\zeta_{i_j}(1+it_j)=
x_0(1+\varepsilon)e^{\frac{2i\pi}{n+1}i_j}(1+it_j)$ for some $t_j\in[-\eta,\eta]$. 

We first assume that the sequence $(i_j)_{1\le j\le m}$ is not constant. Since the unit disc is strictly convex, we have
$$\alpha\coloneqq\sup\left|\frac1m \sum_{p=1}^me^{\frac{2i\pi}{n+1} j_p}\right|<1$$
where the upper bound is taken over all the $m$-tuples of integers $(j_1,\ldots, j_m)$ lying between $0$ and $n$ which are not all equal. Hence if $\varepsilon$ is sufficiently small 
$$\left|\frac{\zeta_{i_1}+\cdots +\zeta_{i_m}}m\right|
\leq \alpha x_0(1+\varepsilon)<x_0.$$ By continuity, if $\eta$ is sufficiently small (depending on $\varepsilon$ but not on the $\xi_j$'s),
we have $$\left|\frac{\xi_1+\cdots +\xi_m}m\right|
<x_0,$$ and as before this yields that
$$\left|\phi\left(\frac{\xi_1+\cdots+\xi_m}m\right)\right|<1\leq \left|\phi(\xi_1)\right|^{\frac1m}\cdots \left|\phi(\xi_m)\right|^{\frac1m}.$$
\par\smallskip
{Next, we assume that} the sequence $(i_j)_{1\le j\le m}$ is constant, i.e.\ $i_j=i_1$ for every $1\le j\le m$, and we consider the function $\psi$ defined on $\R$ by $\psi(t)= \ln\left| \phi\left(\zeta_{i_1}(1+it)\right)\right|=\frac{n+1}{2}\ln (1+t^2)+(n+1)\ln(1+\varepsilon)$. We have $\psi''(t)=(n+1) \frac{1-t^2}{(1+t^2)^2}$ and if $0<\eta<1$, the function $\psi$ is strictly convex on $[-\eta,\eta]$. It follows that
\begin{align*}
 \left|\phi\left(\frac{\xi_1+\cdots+\xi_m}m\right)\right| &< \left|\phi(\xi_1)\right|^{\frac1m}\cdots \left|\phi(\xi_m)\right|^{\frac1m},
\end{align*}
and so (\ref{(ii)}) is shown in this case as well.
\par\smallskip
We now establish (\ref{(i)}). Provided $\varepsilon$ and $\eta$ are sufficiently small, we have
\begin{align*}
 \left|\tilde \xi_1+\cdots+\tilde \xi_{j-d}+ \frac{\xi_1+\cdots+\xi_d}m  \right|
 &\leq (j-d)\eta+\frac{dx_0(1+\varepsilon)}m\\
 &\leq m\eta+\frac{m-1}mx_0(1+\varepsilon) <x_0
\end{align*}
and (\ref{(i)}) readily follows from these inequalities.
\end{proof}

Next, we finally prove that the operator $A$ satisfies the hypothesis of Lemma \ref{lemme1.6}.

\begin{lemma}\label{mainlemma}
Let $B$ be a continuous linear operator on $\oc$ such that there exists $n\ge 0$ which satisfies $B=T_n B T_n$.
We suppose that the sequence
$(B^m z^j)_{m}$ converges to $0$ in $\oc$ as $m$ tends to infinity for every $j\ge 0$. Let $\delta\in (0,1)$ and we set $A=B+\delta S_{n+1}$. 
Then $A$ fulfills the assumptions of Lemma \ref{lemme1.6}.
\end{lemma}

\begin{proof} Let $\cu$ and $\cv$ be two open sets in $\oc$, let $\cw$ be a neighborhood of $0$ in $\oc$ and let $m$ be a positive integer. Let $\ca$ and $\cb$ be the subsets of $\C$ given by Lemma \ref{lemmephi}. Since $\ca$ has accumulation points in $\C$, the linear vector space $\vspan{e_\alpha\ ; \ \alpha\in\ca}$ is dense in $\oc$. Thus there exist $a_1,\ldots, a_p\in\cc$ and $\alpha_1,\ldots, \alpha_p\in\ca$ such that the function $\sum_{j=1}^ pa_je_{\alpha_j}$ belongs to $\cu$.
\par\smallskip
Since $\cb$ is invariant by the rotation of angle $\frac{2\pi}{n+1}$, and since $\cb$ has accumulation points, Lemma \ref{lemme3} asserts that $\vspan{\varepsilon_\beta\ ; \ \beta\in\cb}$ is dense in $\oc$. Thus there exist $b_1,\ldots, b_q\in\cc$ and $\beta_1,\ldots, \beta_q\in\cb$ such that the function $\sum_{j=1}^ qb_j\varepsilon_{\beta_j}$ belongs to $\cv$.
\par\smallskip
During the forthcoming computations, quantities of the form $$\phi\left(\sum_{j=0}^p r_j\alpha_j+\frac1m \sum_{j=0}^q s_j\beta_j\right)$$ will appear, where the  $r_j$'s and $s_j$'s are positive integers such that $\sum_{j=1}^pr_j+\sum_{j=1}^q s_j\leq m$. We will need to be able to ensure that the sums $\sum_{j=0}^p r_j\alpha_j+\frac1m \sum_{j=0}^q s_j\beta_j$  do not belong to the spectrum of $B$. In order to ensure this condition, we notice that for all tuples $r=(r_1,\ldots,r_p)\in\nn^p$ and $s=(s_1,\ldots, s_q)\in\nn^q$ such that $\sum_{j=1}^p r_j+\sum_{j=1}^q s_j \leq m$, the set
$$E_{r,s}=\left\{
\left((\alpha_1,\ldots, \alpha_p),(\beta_1,\ldots,\beta_q)\right)\in\ca^p\times\cb^q\ ; \  
\phi\left(\sum_{j=0}^p r_j\alpha_j+\frac1m \sum_{j=0}^q s_j\beta_j\right) \in\spectre{B}
\right\}$$
is a closed subset of $\ca^p\times\cb^q$ with empty interior. Since  $\sum_{j=1}^p r_j+\sum_{j=1}^q s_j \leq m$, there are only finitely many such sets. Therefore, the set 

\begin{equation} \label{ensemble}
\ca^p\times\cb^q\setminus \left(\bigcup_{\genfrac{}{}{0pt}{}{r\in\nn^p,s\in\nn^q}{\sum_{j=1}^pr_j+\sum_{j=1}^qs_j\leq m}} E_{r,s}\right)
\end{equation}
 is dense in $\ca^p\times\cb^q$.
 \par\medskip
 
We now observe that $e_{\alpha'}$ tends to $e_\alpha$ in $\oc$ as $\alpha'$ tends to $\alpha$, and that $\varepsilon_{\beta'}$ tends to $\varepsilon_\beta$ in $\oc$ as $\beta'$ tends to $\beta$, with $\beta,\beta'\in\cc\setminus \phi^{-1}(\spectre{B})$. The first assertion is obvious. In order to prove the second assertion, let $r$ be a positive integer.
We equip $\cc_n[z]$ with the norm $N_r$ and denote by $\cn_r$ the norm on the space $\cb(\cc_n[z])$ of linear continuous operators  on  $\cc_n[z]$  induced by $N_r$.
We notice that for every $w\in\cc\setminus \spectre{B}$, for every polynomial $P$, $(wI-B)^{-1}P$ is still a polynomial, and that
the map $\Theta : w\mapsto (wI-B)^{-1}B$ from $\cc\setminus\spectre{B}$ into $ \cb(\cc_n[z])$ is well defined and continuous. 

Now for all $\beta,\beta'\in\cc\setminus \phi^{-1}(\spectre{B})$, we have
\begin{align*}
    N_r\left(\varepsilon_{\beta'}-\varepsilon_{\beta}\right)
    =& N_r\left(e_{\beta'}-e_\beta + \Theta(\phi(\beta'))T_n e_{\beta'} -  \Theta(\phi(\beta))T_n e_{\beta} \right)\\
    \leq& N_r\left(e_{\beta'}-e_\beta\right)
    +\cn_r\left(\Theta(\phi(\beta'))-\Theta(\phi(\beta))\right) N_r\left(e_{\beta'}\right)\\
    &+\cn_r\left(\Theta(\phi(\beta)\right) N_r\left(e_{\beta'}-e_\beta\right).
\end{align*}
Therefore,  the continuity of $\Theta$ implies that 
 $N_r\left(\varepsilon_{\beta'}-\varepsilon_{\beta}\right)$ tends to zero as  $\beta'\in\cc\setminus \phi^{-1}(\spectre{B})$ tends to $\beta\in\cc\setminus \phi^{-1}(\spectre{B})$.
 Hence, $\varepsilon_{\beta'}$ tends to $\varepsilon_\beta$ in $\oc$ when $\beta'\in\cc\setminus \phi^{-1}(\spectre{B})$ tends to $\beta\in\cc\setminus \phi^{-1}(\spectre{B})$.
\par\medskip
Therefore, by  choosing $\left((\alpha'_1,\ldots, \alpha'_p),(\beta'_1,\ldots,\beta'_q)\right)$ in the set defined in (\ref{ensemble}) close enough to 
$\left((\alpha_1,\ldots, \alpha_p),(\beta_1,\ldots,\beta_q)\right)$, we can assume that for every $r\in\nn^p$ and every $s\in\nn^q$ such that ${\sum_{j=1}^pr_j+\sum_{j=1}^qs_j\leq m}$, the complex number 
$\phi\left(\sum_{j=0}^p r_j\alpha_j+\frac1m \sum_{j=0}^q s_j\beta_j\right)$ does not belong to $\spectre{B}$.

\par\smallskip
For $N\ge 0$ and $j=1,\ldots, q$, let $c_j(N)$ be a complex number such that $$c_j(N)^m=\frac{b_j}{\phi(\beta_j)^N}\cdot$$ We set $$u_N=\sum_{l=1}^p a_le_{\alpha_l} + \sum_{j=1}^q c_j(N) e_{\frac{\beta_j}m}.$$
\par\smallskip
For every $\beta\in\cb$, we have $|\beta|>x_0$, thus $|\phi(\beta)|>1$ and the sequence $$\left( \sum_{j=1}^q c_j(N) e_{\frac{\beta_j}m}\right)_N$$ converges to $0$ in $\oc$ as $N$ goes to infinity. Consequently, $u_N$ belongs to $\cu$ if $N$ is large enough.
\par\smallskip
We now compute the quantities $u_N^k$ and $A^Nu_N^k$ for  every $k\in\{1,\ldots, m\}$. For every $d$-tuple $L=(l_1,\ldots, l_d)\in \{1,\ldots, p\}^d$, we put $a_L=a_{l_1}\cdots a_{l_d}$ when $d\geq 1$ and $a_\emptyset =1$. For every $J=(j_1,\ldots, j_{k-d})\in \{1,\ldots, q\}^{k-d},$ we put $c_J(N)= c_{j_1}(N)\cdots c_{j_{k-q}}(N) $ when $k-d\geq 1$, and $c_\emptyset(N)=1$. We have
\begin{align*}
 u_N^k&=\sum_{d=0}^k \sum_{\genfrac{}{}{0pt}{}{L\in\{1,\ldots,p\}^{k-d}}{J\in\{1,\ldots,q\}^{d}}} h(L,J) a_Lc_J(N) e_{\alpha_{l_1}+\cdots+\alpha_{l_{k-d}}+\frac1m(\beta_{j_{1}}+\cdots+\beta_{j_{d}})}
\end{align*}
where the $h(L,J)$'s are  constants depending only on $L$ and $J$ but not on $N$ (cf.\ the proof of \cite{Bay1}*{Theorem 2.1}).
\par\smallskip
Setting $\lambda={\alpha_{l_1}+\cdots+\alpha_{l_{k-d}}+\frac1m(\beta_{j_{1}}+\cdots+\beta_{j_{d}}})$, we get from Fact \ref{faitIV-1} that
$$c_J(N) A^Ne_\lambda= c_J(N) \left(\phi(\lambda)^N \varepsilon_\lambda-\left(\phi(\lambda)I-B\right)^{-1}B^{N+1}e_\lambda\right).$$

\begin{enumerate}[(a), itemsep=1ex]
 \item If $k<m$, for every $d\in\{0,\ldots, k\}$ it follows from (\ref{(i)}) that $|\phi(\lambda)|<1$. Since the sequence $(c_J(N))_N$ is bounded, it follows that $\left(c_J(N)\phi(\lambda)^N\varepsilon_\lambda\right)_N$ converges to $0$ in $\oc$. Since $\left(B^{N+1}e_\lambda\right)_N$ also converges to $0$ in $\oc$, $\left( c_J(N) A^Ne_\lambda\right)_N$ converges to $0$, and so there exists $N$ sufficiently large such that for every $k\in\{1,\ldots, m-1\}$, the function $A^Nu_N^k$ belongs to $\cw$.
 
 \item \label{(b)} If $k=m$ and $d<m$, again from (\ref{(i)}) in Lemma \ref{lemmephi} we have $|\phi(\lambda)|<1$, and since $(c_J(N))_N$ is a bounded sequence, $\left( c_J(N) A^Ne_\lambda\right)_N$ converges to $0$ in $\oc$.
 
 \item \label{(c)} If $k=d=m$ and if the $\beta_j$'s are not all equal, then
 \begin{align*}
  \left|c_J(N) \phi(\lambda)^N\right| &= \left|b_{j_1}\cdots b_{j_m}\right|^{\frac1m} 
  \left|
  \frac{\phi\left(\frac{\beta_{j_1}+\cdots+\beta_{j_m}}{m}\right)}
  {\phi\left(\beta_{j_1}\right)^{\frac1m}\cdots \phi\left(\beta_{j_m}\right)^{\frac1m}}
  \right|^{N}.
 \end{align*}
 Assertion (\ref{(ii)}) in Lemma \ref{lemmephi} yields that
 $$\left|
  \frac{\phi\left(\frac{\beta_{j_1}+\cdots+\beta_{j_m}}{m}\right)}
  {\phi\left(\beta_{j_1}\right)^{\frac1m}\cdots \phi\left(\beta_{j_m}\right)^{\frac1m}}
  \right|<1,$$ 
  so that $(c_J(N) \phi(\lambda)^N \varepsilon_\lambda)_N$ converges to $0$ in $\oc$.
  Since $\left(B^{n+1}e_\lambda\right)_N$ converges to $0$ in $\oc$, it follows again that $\left( c_J(N) A^Ne_\lambda\right)_N$ converges to $0$ in $\oc$.
  
  \item \label{(d)} Finally, if $k=m=d$ and if all the $\beta_j$'s are equal, then $\lambda=\beta_{j_1}$, $J=(j_1,\ldots, j_1)$ and 
  \begin{align*}
   c_J(N)A^Ne_\lambda
   &=c_{j_1}(N)^m \phi\left(\beta_{j_1}\right)^N \varepsilon_{\beta_{j_1}} - \left(\phi(\beta_{j_1})I-B\right)^{-1} B^{N+1} e_{\beta_{j_1}}\\
   &=b_{j_1}\varepsilon_{j_1}-\left(\phi\left(\beta_{j_1}\right)I-B\right)^{-1} B^{N+1} e_{\beta_{j_1}}.
  \end{align*}
Since $(B^{N+1} e_{\beta_{j_1}})_N$ converges to $0$ in $\oc$, it follows that $\left( c_J(N)A^Ne_\lambda\right)_N$ converges to $b_{j_1} \varepsilon_{j_1}$.
\end{enumerate}
\par\smallskip
From assertions \ref{(b)}-\ref{(d)} above, we finally deduce that the sequence $\left(A^Nu_N^m)\right)_N$ converges to $\sum_{j=1}^q b_j \varepsilon_j$, and so there exists $N$ large enough such that $A^Nu_N^m$
belongs to $\cv$. This ends the proof of Lemma \ref{mainlemma}.
\end{proof}
\par\medskip
\begin{proof}[Proof of Theorem \ref{maintheorem}:]
  Let $(\cu_p)_{p\geq 1}$ be a basis of open sets of $\oc$, and let $(\cw_r)_{r\geq 1}$ be a basis of neighborhoods of $0$ in $\oc$. Now let
\begin{align*}
\cg\coloneqq& \left\{\vphantom{T^N(u^n) }
A\in \clmk\ ; \ \forall p,q,r\geq 1, \forall m\geq 1, \exists u\in\cu_p,\exists N\geq 1,\right.\\ &\hskip 20pt \left.\text{ such that } \forall 0\leq n<m, T^N(u^n) \in \cw_r \text{ and }T^N(u^m)\in\cu_q
\right\}.
\end{align*}

Therefore, in order to prove Theorem \ref{maintheorem}, it suffices to prove that $\cg$ is a dense  $G_\delta$ set.
\par\smallskip
We can rewrite $\cg$ as 
$$\cg= \bigcap_{p,q,r,m\geq 1}\bigcup_{\genfrac{}{}{0pt}{}{u\in\cu_p}{N\geq 1}}\cg_{p,q,r,m,u,N}
 ,$$
 where
$$\cg_{p,q,r,m,u,N}\coloneqq \{T\in\clmk\ ; \ \forall 0\leq n<m, \; T^N u^m\in\cw_r\text{ and } T^Nu^m\in\cu_q\}.$$ It suffices to show that $\cg_{p,q,r,m,u,N}$ is an open set in $\clmk$.
Let $T_0$ belong to $\cg_{p,q,r,m,u,N}$. If $T$ is sufficiently \sot-close to $T_0$, then Claim \ref{claim} implies that for every $0\leq n<m$, $T^N u^n$ belongs to $\cw_r$ and $T^N u^m$ belongs to $\cu_q$. Thus $\cg_{p,q,r,m,u,N}$ is indeed \sot-open and $\cg$ is a $G_\delta$ set. 
\par\smallskip

We now check that $\cg$ is dense in $\clmk$. Let $A$ belong to $\clmk$ and let $\cv$ be a \sot-neighborhood of $A$  in $\clmk.$
It follows from Corollary \ref{corollaire} that there exist $B\in\clmk$, $\delta\in(0,1)$ and $n\geq 1$ such that $B+\delta S_{n+1}$ belongs to $\cv$, $B=T_nBT_n$ and $B^{n+1}=0$. By Lemma \ref{mainlemma}, the operator $B+\delta S_{n+1}$  satisfies the hypothesis of Lemma \ref{lemme1.6} and thus belongs to $\cg$. Hence $\cg$ is dense in $\clmk$.
This concludes the proof of Theorem \ref{maintheorem}.

\end{proof}

\section{Proof of Theorem \ref{Theorem 2} and related results}\label{section5}
 In this section, the setting is the following: $X$ is one of the (complex) Banach spaces $\ell_{p}(\N)$, $1\le p<+\infty$, or $c_{0}(\N)$. We endow $X$ with the coordinatewise product, so that it 
becomes a Banach algebra. For every $M>1$, a typical operator $T\in(\bmx,\sot)$ is hypercyclic, see the proof of \cite{GMM1}*{Proposition 2.3} which can be adapted in a straightforward way from the Hilbertian setting to the case where $X=\ell_{p}(\N)$, $1\le p<+\infty$, or $X=c_{0}(\N)$. Theorem \ref{Theorem 2} augments this statement by showing that, in fact, a typical $T$ admits a hypercyclic algebra.
\begin{proof}[Proof of Theorem \ref{Theorem 2}]
Let $(U_{q})_{q\ge 1}$ be a basis of the topology of $X$ consisting of open balls. We write $U_{q}=B(x_{q},\rho_{q})$ for each $q\ge 1$, where $x_{q}$ is the center of the ball and $\rho_{q}$ its radius. 
For each $s\ge 1$, we set $W_{s}=B(0,2^{-s})$, so that $(W_{s})_{s\ge 1}$ is a basis of neighborhoods of $0$ in $X$. {Set also $\mathcal{U}_{{q}}=B(x_{q},\frac{1}{2}\rho_{q})$ and $\mathcal{W}_{{s}}=B(0,2^{-(s+1)})$, so that $\overline{\mathcal{U}}_{q}\subseteq {U}_{q}$ and $\overline{\mathcal{W}}_{s}\subset {W}_{s}$.} Let 
$
\bigl(U_{q_{l}},U_{r_{l}},W_{s_{l}},m_{_{0,l}},m_{_{1,l}}\bigr)_{l\ge 1}
$
be an enumeration of all the tuples $(U,V,W,m_{_{0}},m_{_{1}})$, where $U$ and $V$ belong to the set $\{U_{q}\;;\;q\ge 1\}$, $W$ belongs to the set 
$\{W_{s}\;;\;s\ge 1\}$, and $(m_{_{0}},m_{_{1}})$ is a pair of positive integers with $m_{_{0}}<m_{_{1}}$.
\par\medskip
Proceeding as in the proof of Theorem \ref{maintheorem}, we consider the set 
\begin{align*}
 G_{M}(X)=\bigl\{
 T\in \bmx\;&;\forall l\ge 1,\;\exists u_{l}\in U_{q_{l}},\; \exists k_{l}\ge 1\;\textrm{such that }
 T^{k_{l}}\bigl(u_{l}^{m_{_{0,l}}}\bigr)\in U_{r_{l}}\\
 &\textrm{and } T^{k_{l}}\bigl(u_{l}^{n}\bigr)\in W_{s_{l}}
  \;\textrm{for every}\ n\ \textrm{with}\ m_{_{0,l}}<n\le m_{_{1,l}}
  \bigr\}\cdot
\end{align*}
Using the fact that all the maps $T\longmapsto T^{k}$, $k\ge 1$, from $(\bmx,\sot)$ into $\bx$ are continuous \cite{GMM1}*{Lemma 2.1}, it is easy to show that $G_{M}(X)$ is a
$G_{\delta }$ set in $(\bmx,\sot)$. By \cite{BCP}*{Corollary 2.4} (or by Lemma \ref{lemme1.6} in Section \ref{section4} above), every operator belonging to
$G_{M}(X)$ admits a hypercyclic algebra. In order to show that the property of having a hypercyclic algebra is typical in $(\bmx,\sot)$, it thus suffices to show that $G_{M}(X)$ is $\sot$-dense in 
$\bmx$. 
\par\smallskip
Let $d$ be a distance on $\bmx$ which induces the \sot\ and turns $\bmx$ into a complete separable metric space. The main step in order to show the density of $G_{M}(X)$ in $(\bmx,\sot)$ is the following lemma.

\begin{lemma}\label{Lemme 4 refait}
 Let $S\in\bmx$ with $\norm{S}<M$, and let $\varepsilon >0$. Let also  $\mathcal{U}$ and $\mathcal{V}$ be two nonempty open subsets of $X$, let $\mathcal{W}$ be an open neighborhood of $0$, and let $m_{_{0}}<m_{_{1}}$ be two positive integers.
 Then there exists an operator $T\in\bmx$ with $\norm{T}<M$, an integer $k\ge 1$ and a vector $b\in X$ such that:
\begin{enumerate}[\normalfont (a)]
 \item  $d(T,S) <\varepsilon$; 
 \item  $b\in \mathcal{U}$;
 \item  $\smash[t]{T^{k}\bigl(b^{m_{_{0}}}\bigr)\in \mathcal{V}}$ and
 $\;\forall\,m_{_{0}}<n\le m_{_{1}}$, $\smash[t]{T^{k}\bigl(b^{n}\bigr)\in {\mathcal{W}}}$.
 \end{enumerate}
\end{lemma}

\begin{proof}
Denote by $(e_{n})_{n\ge 0}$ the canonical basis of $X$, and  let $E_{N}\coloneqq[e_{0},e_{1},\dots,e_{N}]$ be the linear span of the first $N+1$ basis vectors in $X$, $N\ge 0$. Let also $P_N$ denote the canonical projection of $X$ onto $E_N$, and set $S_N=P_NSP_N$. We have $\norm{S_N}\le\norm{S}<M$.
Let $\delta >0$ and $\gamma >0$ be such that $\max(1,\norm{S}+\delta )<M-\gamma <M$. 

Setting $M' \coloneqq M-\gamma $, we consider the operator $T$ acting on $X$ defined by
\[
Te_{n}=
\begin{cases}
 S_N e_{n}, &\textrm{for }\ 0\le n\le N\\
 \delta e_{n-(N+1)}, &\textrm{for }\ N+1\le n\le 2N+1\\
 M'e_{n-(N+1)}, &\textrm{for }\ n>2N+1.
\end{cases}
\]
Then $T$ is a bounded operator on $X$ with $\norm{T} =M'<M$. Moreover, $T$ can be made as close to $S$ as we wish with respect to the \sot, provided that $N$ is chosen sufficiently large. We fix $N$ so large that $d(T,S) <\dfrac{\varepsilon}{2}\cdot$
\par\smallskip
Operators of this form were introduced in \cite{GMM1}*{Proposition 2.10}, where it is
shown that for every complex number $\lambda $ with
$|\lambda |<M'$, the eigenspace $\ker(T-\lambda I)$ is $(N+1)$-dimensional. Whenever $\Lambda $ is a subset of the open disk $D(0,M')\subset \C$, which has an accumulation point  in 
$D(0,M')$, the eigenspaces $\ker(T-\lambda I)$, $\lambda \in\Lambda $, span a dense linear subspace of $X$. In particular, the vector space
\[
H_{-}(T) \coloneqq \textrm{span}\,[\,\ker(T-\lambda I)\;;\;|\lambda |<1]
\]
is dense in $X$. For every $k\ge 1$ and every $n\ge 0$, we have:
\begin{equation}\label{Equation 1}
 T^{k}e_{k(N+1)+n}=
\begin{cases}
 M'{}^{k-1}\delta e_{n}, &\textrm{if}\ 0\le n\le N\\
 M'{}^{k}e_{n}, &\textrm{if}\ n>N.
\end{cases}
\end{equation}

We choose a  vector $z\in \mathcal{U}$, with finite support, which we write as 
$z=\sum_{j=0}^{J}\gamma _{j}e_{j}$. Since 
$H_{-}(T)$ is dense in $X$, there exists for each $0\le j\le J$ a vector $f_{j}\in H_{-}(T)$ with $\norm{f_{j}-e_{j}} <\frac{\delta _{1}}{J+1}$, where 
$0<\delta _{1}<\varepsilon$ is sufficiently small (how small will be specified in the sequel). 
Letting $u=\sum_{j=0}^{J}\gamma _{j}f_{j}$, for $\delta _{1}$ sufficiently small, 
$u$ belongs to $\mathcal{U}$. We now set for every $n\ge 1$
\begin{equation}\label{Equation 2}
 \ti{u}^{*n}\coloneqq\sum_{j-0}^{J}\gamma _{j}^{n}f_{j}.
\end{equation}
Observe that $\ti{u}^{*n}$ belongs to $H_{-}(T)$ for every $n\ge 1$, so that 
\begin{equation*} %\label{Equation 3} unused label
 \norm{T^{k}\ti{u}^{*n}} \longrightarrow 0\quad\textrm{as}\quad k\longrightarrow +\infty.
\end{equation*}
We have $\ti{u}^{*1}=u$.
\par\medskip 
Let now $v$ be a vector belonging to $\mathcal{V}$, of the form $v=\sum_{i=0}^{d}\alpha _{i}e_{i}$, with $d>N$. For every $k\ge 1$, we set
\begin{equation} \label{Equation 4}
v^{(k)}\coloneqq\sum_{i=0}^{N}\Bigl(\dfrac{\alpha _{i}}{M'{}^{k-1}\delta}\Bigr)^{1/m_{_{0}}}e_{k(N+1)+i} +\sum_{i=N+1}^{d}\Bigl(\dfrac{\alpha _{i}}{M'{}^{k}}\Bigr)^{1/m_{_{0}}}e_{k(N+1)+i}.
\end{equation}
Since $M'>1$, we have
\begin{equation*} %\label{Equation 4 prime} unused label
 \norm{v^{(k)}} \longrightarrow 0\quad\textrm{as}\quad k\longrightarrow+\infty.
\end{equation*}
For  $n = m_{_{0}}, m_{_{0}} +1, \dots, m_{_{1}}$ we have
\begin{equation*} %\label{Equation 6} unused label
 \bigl[v^{(k)}\bigr]^{n}=\sum_{i=0}^{N}\Bigl(\dfrac{\alpha _{i}}{M'{}^{k-1}\delta }\Bigr)^{n/m_{_{0}}}e_{k(N+1)+i}+
 \sum_{i=N+1}^{d}\Bigl(\dfrac{\alpha _{i}}{M'{}^{k}}\Bigr)^{n/m_{_{0}}}e_{k(N+1)+i}
\end{equation*}
(recall that the product considered on $X$ is the coordinatewise product with respect to the canonical basis). We thus have $T^{k}\bigl[v^{(k)}\bigr]^{m_{_{0}}}=v$ by (\ref{Equation 1}). Since 
$M'>1$, we also have 
\begin{equation*}  %\label{Equation 6 bis}
 \norm{\bigl[v^{(k)}\bigr]^{m_{_{0}}}} \longrightarrow 0\quad\textrm{as}\quad k\longrightarrow+\infty.
\end{equation*}
On the other hand, if $m_{_{0}}<n\le m_{_{1}}$, we have
\[
T^{k}\bigl[v^{(k)}\bigr]^{n}=\sum_{i=0}^{N}\Bigl(\dfrac{\alpha _{i}}{M'{}^{k-1}\delta }\Bigr)^{n/m_{_{0}}}M'{}^{k-1}\delta \,e_{i}+\sum_{i=k+1}^{d}
\Bigl(\dfrac{\alpha _{i}}{M'{}^{k}}\Bigr)^{n/m_{_{0}}}M'{}^{k}\,e_{i}
\]
by (\ref{Equation 1}). So
\[
T^{k}\bigl[v^{(k)}\bigr]^{n}=\sum_{i=0}^{N}\alpha _{i}^{n/m_{_{0}}}\cdot\dfrac{1}{\bigl(M'{}^{k-1}\delta \bigr)^{n/m_{_{0}}-1}}\cdot e_{i}+\sum_{i=N+1}^{d}\alpha 
_{i}^{n/m_{_{0}}}\cdot\dfrac{1}{\bigl(M'{}^{k}\bigr)^{n/m_{_{0}}-1}}\cdot e_{i}
\]
and since $\frac{n}{m_{_{0}}}-1>0$, we obtain that
\begin{equation*} %\label{Equation 7}
 \norm{T^{k}\bigl[v^{(k)}\bigr]^{n}} \longrightarrow 0\quad\textrm{as}\quad k\longrightarrow+\infty\quad\textrm{for every}\quad m_{_{0}}<n\le m_{_{1}}.
\end{equation*}
Summarising, we have shown that 
\begin{align}
 {u}+v^{(k)} &\longrightarrow u &&\textrm{as }\, k\longrightarrow+\infty\label{Equation 8}\\
 T^{k}\bigl(\ti{u}^{*m_{_{0}}}+\bigl[v^{(k)}\bigr]^{m_{_{0}}}\bigr) &\longrightarrow v&&\textrm{as }\, k\longrightarrow+\infty\label{Equation 9}\\
 T^{k}\bigl(\ti{u}^{*n}+\bigl[v^{(k)}\bigr]^{n}\bigr) &\longrightarrow 0&&\textrm{as }\, k\longrightarrow+\infty\quad \textrm{for every }\, m_{_{0}}<n\le m_{_{1}}.\label{Equation 10}
\end{align}
\par\medskip
We now define a bounded  operator $L$ on $X$ by setting
\[
Le_{j}=
\begin{cases}
 f_{j}, &\textrm{if}\ 0\le j\le J\\
 e_{j}, &\textrm{if}\ j>J.
\end{cases}
\]
Then we have 
\[
\norm{Lx-x} = \norm{\sum_{j=0}^{J}x_{j}(f_{j}-e_{j})} \le\max_{0\le j\le J}|x_{j}|
\sum_{j=0}^{J} \norm{f_{j}-e_{j}} \le \delta _{1} \cdot \norm{x} \quad \textrm{for every}\ x\in 
X,\] so that $\norm{L-I} \le\delta _{1}<\varepsilon$.  If $0<\delta_1 <1$, $L$ is invertible, and
\[
\norm{L^{-1}-I} \le\dfrac{\delta _{1}}{1-\delta _{1}}\le 2\delta _{1}
\]
as soon as $0<\delta _{1}<1/2$. Let now $x_{k}={u}+v^{(k)}$ and $y_{k}=L^{-1}({u}+v^{(k)})$. If $k$ is sufficiently large, the support of 
$v^{(k)}$ is disjoint from the interval $[0,J]$. We have
\begin{align*}
 y_{k}=\sum_{j=0}^{J}\gamma _{j}e_{j}&+\sum_{i=0}^{N}\Bigl(\dfrac{\alpha _{i}}{M'{}^{k-1}\delta }\Bigr)^{1/m_{_{0}}}e_{k(N+1)+i}
 +\sum_{i=N+1}^{d}\Bigl(\dfrac{\alpha _{i}}{M'{}^{k}}\Bigr)^{1/m_{_{0}}}e_{k(N+1)+i}
\end{align*}
so that for every $n\ge 1$, 
\begin{align*}
 y_{k}^{n}=\sum_{j=0}^{J}\gamma _{j}^{n}e_{j}&+\sum_{i=0}^{N}\Bigl(\dfrac{\alpha _{i}}{M'{}^{k-1}\delta }\Bigr)^{n/m_{_{0}}}e_{k(N+1)+i}
 +\sum_{i=N+1}^{d}\Bigl(\dfrac{\alpha _{i}}{M'{}^{k}}\Bigr)^{n/m_{_{0}}}e_{k(N+1)+i}.
\end{align*}
It follows that
\begin{equation}\label{Equation 10 prime}
 L(y_{k}^{n})=\ti{u}^{*n}+\bigl[v^{(k)}\bigr]^{n}\quad\textrm{for all}\ k\ \textrm{sufficiently large.}
\end{equation}
By  (\ref{Equation 8}), (\ref{Equation 9}), and (\ref{Equation 10}), combined with (\ref{Equation 10 prime}), we have, as $k$ goes to infinity,
\begin{align}
 x_{k}\longrightarrow u \quad &\textrm{and}\quad y_{k}\longrightarrow L^{-1}u\label{Equation 11}\\
 T^{k}L(y_{k}^{m_{_{0}}}) &\longrightarrow v \nonumber \\   % \label{Equation 12} unused label
 T^{k}L(y_{k}^{n}) &\longrightarrow 0\quad\textrm{for every }\, m_{_{0}}<n\le m_{_{1}}. \nonumber %\label{Equation 13}
\end{align}
Set now $T_{1}\coloneqq L^{-1}TL$. Then
{\begin{eqnarray*}
 \norm{T_{1}-T}&=& \norm{L^{-1}TL-T}\le \norm{L^{-1}T(L-I)}+ \norm{(L^{-1}-I)T}\\
&\le&  \norm{T}\,\delta _{1} (1+2\,\delta _{1})+\norm{T}\,.\, 2\,\delta _{1} <4\norm{T}\,\delta _{1},
\end{eqnarray*}
so that $\norm{T_{1}}<M$ and $d(T_{1},T)<\dfrac{\varepsilon}{2}$, provided that $\delta_1$ is sufficiently small.
We also deduce from the properties above that, as $k$ goes to infinity,
\begin{align}
 T_{1}^{k}(y_{k}^{m_{_{0,1}}})&\longrightarrow L^{-1}v\label{Equation 14}\\
 T_{1}^{k}(y_{k}^{n})&\longrightarrow 0\quad\textrm{for every }\, m_{_{0}}<n\le m_{_{1}}.\label{Equation 15}
\end{align}
Now, $L^{-1}u=z$ belongs to $\mathcal{U}$ and if $\delta _{1}$ is sufficiently small, $L^{-1}v$ belongs to $\mathcal{V}$. If we choose $k$ sufficiently large, and then set
$b=y_{k}$, we eventually obtain that 
$b\in \mathcal{U}$ (by (\ref{Equation 11})),
$T_{1}^{k}\bigl(b^{m_{_{0}}}\bigr)\in \mathcal{V}$ (by (\ref{Equation 14})}), and
$T_{1}^{k}\bigl(b^{n}\bigr)\in \mathcal{W}$ for every $m_{_{0}}<n\le m_{_{1}}$ (by (\ref{Equation 15}))
which is exactly what is required by Lemma \ref{Lemme 4 refait}.
\end{proof}

Returning to the proof of Theorem \ref{Theorem 2}, we now consider an operator $S\in\bmx$ with $\norm{S}<M$, and $\varepsilon>0$. Our aim is to construct an operator $T\in G_{M}(X)$ with $d(T,S)<\varepsilon$.
Let $(\varepsilon _{l})_{l\ge 1}$ be a sequence of positive numbers such that $\sum_{l\ge 1}\varepsilon _{l}<\varepsilon$.
Using Lemma \ref{Lemme 4 refait}, we can construct by induction on $l\ge 1$
 \begin{enumerate}
  \item [$\bullet$] a sequence  $(T_{l})_{l\ge 0}$ of operators on $X$, with $T_0=S$,
  \item [$\bullet$] a sequence $(b_{l})_{l\ge 1}$ of vectors of $X$,
  \item [$\bullet$] a sequence $(k_{l})_{l\ge 1}$ of positive integers
 \end{enumerate}
such that for every $l\ge 1$,
\begin{enumerate}[(a)]
 \item  $d(T_{l},T_{l-1})<\varepsilon_l$;
 \item $b_{l}\in {\mathcal{U}_{q_{l}}}\subseteq {U}_{q_{l}}$;
 \item  $\forall\,1\le l'\le l$,\quad $\smash[t]{T_{l}^{k_{l'}}\bigl(b_{l'}^{m_{_{0,l'}}}\bigr)\in {\mathcal{U}_{r_{l'}}}}$ and
 $\forall\,m_{_{0,l'}}<n\le m_{_{1,l'}}$, $\smash[t]{T_{l}^{k_{l'}}\bigl(b_{l'}^{n}\bigr)\in {\mathcal{W}_{s_{l'}}}}$.
\end{enumerate}
\par\smallskip
The fact that we can ensure, at each step $l\ge 1$, that condition (c) holds for every $1\le l'\le l$, and not only for $l$, relies on the observation that the maps $T\longmapsto T^k$, $k\ge 1$, are \sot-continuous from $\bmx$ into $\bx$. Thus if $T_l$ is constructed sufficiently close to $T_{l-1}$, i.e.\ if $d(T_{l},T_{l-1})$ is sufficiently small, our induction assumption implies that (c) is true for every $1\le l'< l$.
\par\smallskip
Since $(\bmx,d)$ is a complete metric space, the sequence of operators $(T_l)_{l\ge 1}$ converges in $(\bmx,d)$ to an operator $T\in\bmx$ such that $d(T,S)<\varepsilon$. Moreover, letting $l$ tend to infinity in (c) above, and using again the \sot-continuity of the maps $T\longmapsto T^k$, we obtain that for all $l\ge 1$,
 $\smash[t]{T^{k_{l}}\bigl(b_{l}^{m_{_{0,l}}}\bigr)\in \overline{{\mathcal{U}}}_{r_{l}}}\subseteq {{U}_{r_{l}}}$, and
 $\smash[t]{T^{k_{l}}\bigl(b_{l}^{n}\bigr)\in \overline{\mathcal{W}}_{s_{l'}}}\subseteq{W_{s_{l}}}$ for all $m_{_{0,l}}<n\le m_{_{1,l}}$.
Hence $T$ belongs to $G_M(X)$.
\par\smallskip
We have thus proved that $G_{M}(X)$ is dense in $(\bmx,\sot)$. This finishes the proof of Theorem \ref{Theorem 2}.
\end{proof}

There are other interesting topologies which can turn the operator balls $\bmx$ into Polish spaces. One of the most relevant is the so-called \emph{Strong$^{*}$ Operator Topology} (\sote): 
if 
$(T_{\alpha })$ is a net of operators in $\bx$, $T\in\bx$ and $T^* \in \mathcal{B}(X^*)$ is the adjoint of $T$, then we say that  $$\xymatrix{T_{\alpha }\ar[r]^{\sote}&T} \textrm{ if and only if }\xymatrix{T_{\alpha }\ar[r]^{\sot}&T}\textrm{ in }X \textrm{ and }\xymatrix{T_{\alpha 
}^{*}\ar[r]^{\sot}&T^{*}} \textrm{ in } X^{*}.$$

When $X$ is a Banach space with separable dual, the balls $\bmx$, $M>0$, are Polish spaces when endowed with the
\sote\ topology. See the works \cite{EisMat}, \cite{GMM1}, \cite{GMM2}, \cite{GMM3}, and \cite{GriMat} for a study of typical properties of operators 
for the \sote, as well as explanations on the relevance of this topology. 
\par\medskip
Theorem \ref{Theorem 2} admits the following analogue for the \sote\ topology.

\begin{theorem}\label{Theorem 2 bis}
 Let $X=\ell_{p}(\N)$, $1<p<+\infty$, or $X=c_{0}(\N)$, endowed with the coordinatewise product. Let $M>1$. A typical operator $T\in(\bmx,\emph{\sote})$ admits a hypercyclic 
 algebra.
\end{theorem}
\begin{proof}
 The proof is similar to that of Theorem \ref{Theorem 2}. Since the set $G_{M}(X)$ is \sot-$G_{\delta }$, it is also \sote-$G_{\delta }$. 
 And a look at the proof of Theorem \ref{Theorem 2} shows that $G_{M}(X)$ is in fact \sote-dense in $\bmx$.
\end{proof}

We say that  $T \in \bx$ is  \emph{dual hypercyclic} if both $T$ and its adjoint $T^*$ are hypercyclic. The first examples of such operators were obtained by Salas \cites{Salas,Sal07} and  Petersson~\cite{Pet06}, and it is an immediate consequence of the remark after \cite{GMM1}*{Proposition 2.3} (see also \cite{GriMat}*{Fact 2.1}) that when $X=\ell_{p}(\N)$, $1<p<+\infty$, or $X=c_{0}(\N)$, an $\sotstar$-typical  $T \in \Bmx$ is dual hypercyclic.
Using the same argument, we can deduce from Theorem \ref{Theorem 2 bis} the following result.

\begin{proposition}\label{dualhc}
If $X=\ell_{p}(\N)$, $1<p<+\infty$, or $X=c_{0}(\N)$, and $M>1$, a typical operator $T\in(\bmx,\emph{\sote})$ is such that both $T$ and $T^{*}$ admit a 
hypercyclic algebra (with respect to the coordinatewise product).
\end{proposition}

The existence of $T\in\bx$ such that both $T$ and $T^{*}$ admit a hypercyclic algebra can also be 
deduced directly from \cite{BCP}*{Corollary 4.11} and \cite{Salas}.

\section{Typicality of Admitting a Hypercyclic Subspace}\label{section6}

Let $X$ be a complex separable infinite-dimensional Banach space. Our aim in this final section is to study whether, given $M>1$, a typical operator $T\in(\bmx,\sot)$ admits a closed infinite-dimensional hypercyclic subspace. When it makes sense, we also consider this question for $T\in(\bmx,\sotstar)$.
\par\smallskip
For a hypercyclic operator $T \in \bx$, a \emph{hypercyclic subspace} is defined as a closed infinite-dimensional subspace $Z \subseteq X$ such that every nonzero vector in $Z$ is hypercyclic for $T$.
Since the early work of Bernal Gonz\'{a}lez and Montes-Rodr\'{\i}guez \cite{BM95}, the investigation of hypercyclic subspaces has amassed a vast literature, as documented in \cite{BayMat}*{Chapter 8} and \cite{GroPer}*{Chapter 10}.
In contrast to the generic property that the set $\hc{T} \cup \{0\}$ always contains a dense linear manifold of hypercyclic vectors, it turns out that there exist hypercyclic operators that do not admit a hypercyclic subspace. 
Examples of operators that support hypercyclic subspaces in the Fr\'echet space setting include the differentiation and translation operators acting on the space $\entireFns$ of entire functions (cf.\ \cite{GroPer}*{Examples 10.12 and 10.13}).  For the Banach spaces $X=\ell^p(\N)$, $1 \leq p < \infty$, or $X=c_0(\N)$, a weighted backward shift $B_w \in \bx$ admits a hypercyclic subspace if 
\begin{equation*}
\sup_{n\geq 1} \, \prod_{j=1}^n \abs{w_{j}} =+ \infty\;\textrm{ and }\;
\sup_{n\geq 1} \, \limsup_{k \to \infty} \prod_{j=1}^n \abs{w_{j+k}} < +\infty,
\end{equation*}
cf.\ \cite{GroPer}*{Example  10.10}.
However, it is well known that scalar multiples $cB$ of the backward shift for $\abs{c}>1$ do not possess hypercyclic subspaces (cf.\ \cite{GroPer}*{Example 10.26}).  
\par\smallskip
The following characterization of operators that satisfy the Hypercyclicity Criterion and admit a hypercyclic subspace was identified by Le{\'o}n-Saavedra and Montes-Rodr{\'{\i}}guez~\cite{LM01} in the Hilbert space setting, and by Gonz{\'a}lez et al.\ in \cite{GLM00} for separable complex Banach spaces.

\begin{theorem}[\cite{GLM00}, \cite{LM01}] \label{theorem:HcSubspChar}
	Let $X$ be a separable complex Banach space. Suppose that $T \in \bx$  satisfies the Hypercyclicity Criterion.  The following assertions are equivalent.
	\begin{enumerate}[\normalfont (i),  itemsep=.5ex]
		\item $T$ possesses a hypercyclic subspace.
		
		\item There exists some closed infinite-dimensional subspace $Z_0 \subseteq X$ and an increasing sequence of integers $(n_k)$ such that $T^{n_k}x \to 0$ for all $x \in Z_0$.
		
		\item There exists some closed infinite-dimensional subspace $Z_0 \subseteq X$ and an increasing sequence of integers $(n_k)$ such that $\sup_k \lVert T^{n_k}_{\vert Z_0} \rVert < \infty$.
		
		\item The essential spectrum $\essSpec{T}$ of $T$ intersects the closed unit disk.
	\end{enumerate}
\end{theorem}

We recall that an operator $T \in \bx$ is said to be \emph{Fredholm} if both the dimension of its kernel $\dim \ker(T)$ and the codimension of its range $\codim \ran(T)$ are finite. Equivalently, $T$ is Fredholm if and only if it has closed range and $\dim \ker(T) < \infty$ and $\dim \ker(T^*)< \infty$.
The \emph{essential spectrum}  $\essSpec{T}$ of the operator $T \in \bx$ is defined to be
\begin{equation*}
\essSpec{T} \coloneqq \{ \lambda \in \C : T - \lambda \textrm{ is not Fredholm} \}.
\end{equation*}

\par\smallskip

The following theorem is a straightforward consequence of  results from \cite{GMM1} and \cite{GMM2}.

\begin{theorem} \label{theorem:HcSubspTypicalSOT}
	Let $X = \ell_p(\N)$, $1\le p<+\infty$, or $X = c_0(\N)$. For every $M>1$, a typical $T \in (\Bmx, \emph{\sot})$ admits a hypercyclic subspace.
\end{theorem}

\begin{proof}	
An elementary adaptation of the proofs of \cite{GMM1}*{Proposition 2.3} and \cite{GMM1}*{Proposition 2.16}  shows that a typical operator  $T \in (\Bmx, {\sot})$ satisfies the Hypercyclicity Criterion. It thus suffices to show that
 the essential spectrum of a typical  $T \in (\Bmx, {\sot})$ is equal to the closed disk $\overline{D}(0,M)$.
% \par\smallskip
As observed in \cite{GMM2}*{Remark 4.5}, it follows from \cite{EisMat}*{Lemma 5.13} that a typical  $T \in (\Bmx, {\sot})$ is not Fredholm. 
Consider now the set
$$\mathcal{A}\coloneqq\{T\in \Bmx\,;\, \forall\lambda\in D(0,M),\, T-\lambda\textrm{ is surjective}\}.$$
By the topological $0-1$ law \cite{GMM2}*{Proposition 3.2}, $\mathcal{A}$ is either meager or comeager in $(\Bmx, {\sot})$.
\par\smallskip
-- Suppose first that $\mathcal{A}$ is \sot-comeager in $(\Bmx, {\sot})$ (which is known to happen when $X=\ell_1(\N)$ or when $X=\ell_2(\N)$). 
Since a typical  $T \in (\Bmx, {\sot})$ is not Fredholm, the continuity of the Fredholm index on $D(0,M)$ implies that
 for all $\lambda\in D(0,M)$, $T - \lambda$ is not Fredholm (see the argument just before Remark 4.5 in \cite{GMM2}). Hence $\essSpec{T} = \overline{D}(0,M)$.
\par\smallskip
-- Suppose now that $\mathcal{A}$ is \sot-meager in $(\Bmx, {\sot})$ (which is known to happen when $X=\ell_p(\N)$ for $p>2$). As the set of operators $T\in (\Bmx, {\sot})$ such that $T-\lambda$ has dense range for every $\lambda\in D(0,M)$ is comeager in $(\Bmx, {\sot})$ (cf.\ \cite{GMM2}*{Proposition 3.9}), a typical $T \in (\Bmx, {\sot})$ is such that for every $\lambda\in D(0,M)$, $T-\lambda$ does not have closed range. Hence $T - \lambda$ is not Fredholm, and $\essSpec{T} = \overline{D}(0,M)$.
\end{proof}

The following shows that we have a similar statement for $\sotstar$-typical operators.

\begin{theorem} \label{theorem:HcSubspTypicalStar}
Let $X = \ell_p(\N)$, $1< p<+\infty$. For every $M>1$, a typical $T \in (\Bmx, \sotstar)$ admits a hypercyclic subspace.
\end{theorem}

\begin{proof}
This is a direct consequence of the fact that a typical operator  $T \in (\Bmx, {\sotstar})$ satisfies the Hypercyclicity Criterion, and of \cite{GriMat}*{Theorem 3.1}, which states that a typical $T \in (\Bmx, {\sotstar})$ is such that $\essSpec{T} = \overline{D}(0,M)$.	
\end{proof}

We can also deduce from Theorem \ref{theorem:HcSubspTypicalStar} the following result.

\begin{proposition}
	Let $X = \ell_p(\N)$, $1< p<+\infty$. For every $M>1$, a typical operator $T \in (\Bmx, \sotstar)$ is such that $T$ and its adjoint $T^*$ both admit hypercyclic subspaces.
\end{proposition}

\bibliographystyle{siam}
\bibliography{bibliographie}
\end{document}